\documentclass[english]{article}
\usepackage[T1]{fontenc}
\usepackage[latin9]{inputenc}
\usepackage{geometry}
\usepackage{babel}
\usepackage{float}
\usepackage{mathtools}
\usepackage{bm}
\usepackage{amsmath}
\usepackage{amsthm}
\usepackage{amssymb}
\usepackage{wasysym}
\usepackage[unicode=true,pdfusetitle,
 bookmarks=true,bookmarksnumbered=false,bookmarksopen=false,
 breaklinks=false,pdfborder={0 0 1},backref=false,colorlinks=false]
 {hyperref}

\makeatletter

\floatstyle{ruled}
\newfloat{algorithm}{tbp}{loa}
\providecommand{\algorithmname}{Algorithm}
\floatname{algorithm}{\protect\algorithmname}

\newenvironment{lyxcode}
	{\par\begin{list}{}{
		\setlength{\rightmargin}{\leftmargin}
		\setlength{\listparindent}{0pt}% needed for AMS classes
		\raggedright
		\setlength{\itemsep}{0pt}
		\setlength{\parsep}{0pt}
		\normalfont\ttfamily}%
	 \item[]}
	{\end{list}}

\newtheorem{lemma}{\textbf{Lemma}}
\newtheorem{theorem}{\textbf{Theorem}}
\newtheorem{corollary}{\textbf{Corollary}}
\newtheorem{assumption}{\textbf{Assumption}}

\newtheorem{remark}{\textbf{Remark}}
\DeclareMathOperator*{\KL}{KL}

\newcommand{\binm}{C_{i,j}^{m}}
\newcommand{\centerm}{q_{i,j}^{m}}

\newcommand{\instfam}{\mathcal{F}}
\newcommand{\omgdist}{\mathrm{Unif}(\Omega)}
\newcommand{\eoverpi}{\mathbb{E}_{\pi,\omega}}
\newcommand{\wloo}{\omega_{-(i,j)}}
\newcommand{\omgloo}{\Omega_{-(i,j)}}
\newcommand{\pminus}{\mathbb{P}_{\pi,\omega_{i,j}^{m}=-1}}
\newcommand{\pplus}{\mathbb{P}_{\pi,\omega_{i,j}^{m}=1}}

\usepackage{xspace}
\newcommand{\myalg}{\texttt{BaSEDB}\xspace}

\usepackage{color}
\definecolor{cm}{RGB}{0,0,200}

\definecolor{yy}{RGB}{0,200,0}

\definecolor{rj}{RGB}{0,0,200}

\allowdisplaybreaks

\usepackage{tikz}
\usepackage{caption}
\usepackage{algorithmic}
\usepackage{authblk}
\usetikzlibrary{decorations.pathreplacing}
\author[1]{Rong Jiang}
\author[2]{Cong Ma}
\affil[1]{Committee on Computational and Applied Mathematics, University of Chicago}
\affil[2]{Department of Statistics, University of Chicago}

\makeatother

\begin{document}
\title{Batched Nonparametric Contextual Bandits}

\date{}

\maketitle
\begin{abstract}

We study nonparametric contextual bandits under batch constraints,
where the expected reward for each action is modeled as a smooth function of covariates, 
and the policy updates are made at the end of each batch of observations. We establish a minimax regret lower bound for this setting and propose a novel batch learning algorithm that achieves the optimal regret (up to logarithmic factors). 
In essence, our procedure dynamically splits the covariate space into smaller bins, carefully aligning their widths with the batch size. Our theoretical results suggest that for nonparametric contextual bandits, a nearly constant number of policy updates can attain optimal regret in the fully online setting.
 
\end{abstract}

% \tableofcontents

\section{Introduction}

Recent years have witnessed substantial progress in the field of sequential
decision making under uncertainty. Especially noteworthy are the advancements
in personalized decision making, where the decision maker uses side-information
to make customized decision for a user. The contextual bandit framework
has been widely adopted to model such problems because of its applicability
and elegance \cite{li2010contextual,tewari2017ads,bastani2020online}.
In this framework, one interacts with an environment for a number
of rounds: at each round, one is given a context, picks an action,
and receives a reward. One can update the action-assignment policy
based on previous observations and the goal is to maximize the expected
cumulative rewards. For example, in online news recommendation, a
recommendation algorithm selects an article for each newly arrived
user based on the user's contextual information, and observes whether
the user clicks the article or not. The goal is to try to maximize the number
of clicks received. Apart from news recommendation, contextual bandits
have found numerous applications in other fields such as clinical
trials, personalized medicine, and online advertising \cite{kim2011battle,zhou2019tumor,chapelle2014modeling}.

At the core of designing a contextual bandit algorithm is deciding
how to update the policy based on prior observations. A standard metric
of performance for bandit algorithms is regret, which is the
expected difference between the cumulative rewards obtained by an oracle
who knows the optimal action for every context and that obtained by
the actual algorithm under consideration. Many existing regret
optimal bandit algorithms require a policy update per observation
(unit) \cite{auer2002using,abbasi2011improved,perchet2013multi,lattimore2020bandit}.
At a first glance, such frequent policy updates are needed so that
the algorithm can quickly learn the optimal action under each context
and reduce regret. However, this kind of algorithm ignores an important
concern in the practice of sequential decision making---the batch
constraint. 

In many real world scenarios, the data often arrive in batches: the statistician can only observe the outcomes of the policy at the
end of a batch, and then decides what to do for the next batch. For
example, this batch constraint is ubiquitous in clinical trials: statisticians need to divide the participants into batches, determine
a treatment allocation policy before the batch starts, and then observe
all the outcomes at the end of the batch \cite{Robbins1952SomeAO}.
Policy updates are made per batch instead of per unit. In fact, it
is infeasible to apply unit-wise policy update in this case because
observing the effect of a treatment takes time and if one waits for
the result before deciding how to treat the next patient, the entire
experiment will take too long to complete when the number of participants
is huge. The batch constraint also appears in areas such as online
marketing, crowdsourcing, and simulations \cite{bertsimas2007learning,schwartz2017customer,kittur2008crowdsourcing,chick2009economic}.
Clearly, the batch constraint presents additional challenges to online
learning. Indeed, from an information perspective, the statistician's information
set is largely restricted since she can only observe all the responses
at the end of a batch. The following questions naturally arise: 

Given a batch budget $M$ and a total number of $T$ rounds, how
should the statistician determine the size of each batch, and how
should she update the policy after each batch? Can the statistician
design batch learning algorithms that achieve regret performances
on par with the fully online setting using as few policy updates as
possible?

% \textit{Given a batch budget $M$ and a total number of $T$ rounds, how
% should the statistician determine the size of each batch, and how
% should she update the policy after each batch? Can the statistician
% design batch learning algorithms that achieve regret performances
% on par with the fully online setting using as few policy updates as
% possible?}

\subsection{Main contributions}

In this work, we address the aforementioned questions under a classical
framework for personalized decision making---nonparametric contextual
bandits~\cite{rigollet2010nonparametric,perchet2013multi}. In this
framework, the expected reward associated with each treatment (or
arm in the language of bandits) is modeled as a nonparametric smooth function
of the covariates~\cite{yang2002nonp}. In the fully online setup,
seminal works \cite{rigollet2010nonparametric,perchet2013multi} establish
the minimax optimal regret bounds for the nonparametric contextual
bandits. Nevertheless, under the more challenging setting with the
batch constraint, the fundamental limits for nonparametric bandits
remain unknown. Our paper aims to bridge this gap. More concretely,
we make the following contributions:
\begin{itemize}
\item First, we establish a minimax regret lower bound for the nonparametric
bandits with the batch constraint. Our lower bound holds even when the batch size is adaptively chosen (based on the data observed in prior batches). The proof relies on a simple
but useful insight that the worst-case regret over the entire horizon
is greater than the worst-case regret over the first $i$ batches for any $1\le i\le M$. To exploit this insight, for each
different batch, we construct different families of hard
instances to target it, leading to a maximal regret over this
batch. 
% \item Moreover, we prove a lower bound even when the batch size is chosen adaptively, and show adaptivity does not offer a fundamentally better regret.
\item In addition, we demonstrate that the aforementioned lower bound is
tight by providing a matching upper bound (up to log factors). Specifically,
we design a novel algorithm---Batched Successive Elimination with
Dynamic Binning (\myalg)---for the nonparametric bandits with batch constraints.
\myalg \ progressively splits the covariate space into smaller
bins whose widths are carefully selected to align well with the corresponding
batch size. The delicate interplay between the batch size and the
bin width is crucial for obtaining the optimal regret under the batch
setting. 
\item On the other hand, we show the suboptimality of static binning under the batch
constraint by proving an algorithm-specific lower bound. Unlike the
fully online setting where policies that use a fixed number of bins
can attain the optimal regret~\cite{perchet2013multi}, our lower bound indicates that batched successive elimination with static binning is strictly suboptimal.\footnote{In a certain regime the BSE policy from \cite{perchet2013multi} which uses a fixed number bins could loose by log factors compared to the optimal fully online regret. However, we will show the price of fixed binning is polynomial under the batch setting.} This highlights
the necessity of dynamic binning in some sense under the batch setting,
which is uncommon in classical nonparametric estimation. 

\item Last but not least, we demonstrate the challenge of adapting to the margin parameter in tha batch setting. Specifically, we show that when $M$ is small, the price of not knowing the true margin parameter for an algorithm is at least a polynomial increase in terms of the regret. 

\end{itemize}
It is also worth mentioning that
an immediate consequence of our results is that $M\gtrsim\log\log T$
number of batches suffices to achieve the optimal regret in the fully
online setting. In other words, we can use a nearly constant number
of policy updates in practice to achieve the optimal regret obtained
by policies that require one update per round. 

\subsection{Related work}

\paragraph{Nonparametric contextual bandits.} \cite{woodroofe1979one}
introduced the mathematical framework of contextual bandit. The theory
of contextual bandits in the fully online setting has been continuously
developed in the past few decades. On one hand, \cite{auer2002using,abbasi2011improved,goldenshluger2013linear,bastani2020online,bastani2021mostly,qian2023adaptive}
obtained learning guarantees for linear contextual bandits in both
low and high dimensional settings. On the other hand, \cite{yang2002nonp}
introduced the nonparametric approach to model the mean reward function.
\cite{rigollet2010nonparametric} proved a minimax lower bound on
the regret of nonparametric bandit and developed an upper-confidence-bound
(UCB) based policy to achieve a near-optimal rate. \cite{perchet2013multi}
improved this result and proposed the Adaptively Binned Successive
Elimination (ABSE) policy that can also adapt to the unknown margin
parameter. Further insights in this nonparametric setting were developed
in subsequent works \cite{qian2016kernel,qian2016random,reeve2018k,guan2018nonparametric,hu2022smooth,suk2021self,gur2022smoothness,cai2022transfer,suk2023tracking,blanchard2023non}.
The smoothness assumption is also adopted in another line of work
\cite{lu2009showing,locatelli2018adaptivity,krishnamurthy2020contextual,cai2022stochastic}
on the continuum-armed bandit problems. However in contrast to what
we study, the reward is assumed to be a Lipschitz function of the
action, and the covariates are not taken into considerations. 

\paragraph{Batch learning.} The batch constraint has received increasing
attention in recent years. \cite{perchet2016batch,zijun2020batch}
considered the multi-armed bandit problem under the batch setting
and showed that $O(\log\log T)$ batches are adequate in achieving
the rate-optimal regret, compared to the fully online setting. \cite{han2020sequential,ren2020batched}
extended batch learning to the (generalized) linear contextual bandits
and \cite{ren2023dynamic,wang2020online,fan2023provably} further
studied the setting with high-dimensional covariates. \cite{karbasi2021parallelizing,kalkanli2021batched}
established batch learning guarantees for the Thompson sampling algorithm.
\cite{feng2022lipschitz} considered Lipschitz continuum-armed bandit
problem with the batch constraint. Inference for batched bandits was considered in~\cite{zhang2020inference}. A concept related to batch learning
in literature is called delayed feedback \cite{chapelle2011empirical,chapelle2014modeling,vernade2017stochastic,gael2020stochastic}.
These works consider the setting where rewards are observed with delay
and analyze effects of delay on the regret. \cite{kock2017optimal,arya2020randomized}
studied delayed feedback in nonparametric bandits and the key difference
to batch learning is that the batch size is given, whereas in our
case, it is a design choice by the statistician. Batch learning's
focus is different to that of delayed feedback in the sense that the
former gives the decision maker discretion to choose the batch size
which makes it possible to approximate the optimal standard online
regret with a small number of batches. The asymptotic regret of batch learning has been studied in \cite{tianyuan2021anytime, tianyuan2021double} 
% where algorithms that could tackle unknown time horizon were developed for the multi-armed bandit problems. 
and it is shown that the asymptotic optimality can be achieved with a constant number of batches for multi-armed bandits when the time horizon is known.

\paragraph{Switching cost.}
The notion switching
cost is intimately related to the batch constraint. \cite{cesa2013online}
studied online learning with low switching cost and obtained the rate-optimal regret with $O(\log\log T)$ switches. \cite{ruan2021linear} established the minimax optimal regret versus policy switch tradeoff in linear bandits with adverserial and stochastic contexts. \cite{bai2019provably,zhang2020almost,gao2021provably,wang2021provably,qiao2022sample}
developed regret guarantees with low switching cost for reinforcement
learning. \cite{tao2019collaborative} examined collaborative best arm identification under limited interaction, achieving nearly tight tradeoffs between the number of rounds and speedup in multi-armed bandits. Low switching cost can be interpreted as infrequent policy
updates, but it does not necessarily require the learner to divide the samples
into batches with feedback only becoming available at the end of a
batch. 

\section{Problem setup}

We begin by introducing the problem setup for nonparametric bandits
with the batch constraint.

A two-arm nonparametric bandit with horizon $T\geq1$ is specified
by a sequence of independent and identically distributed random vectors
\begin{equation}
(X_{t},Y_{t}^{(1)},Y_{t}^{(-1)}),\qquad\text{for }t=1,2,\ldots,T,\label{eq:bandit}
\end{equation}
where $X_{t}$ is sampled from a distribution $P_{X}$. Throughout
the paper, we assume that $X_{t}\in\mathcal{X}\coloneqq[0,1]^{d}$,
and $P_{X}$ has a density (w.r.t.~the Lebesgue measure) that is
bounded below and above by some constants $\underline{c},\bar{c}>0$,
respectively. For $k\in\{1,-1\}$ and $t\geq1$, we assume that $Y_{t}^{(k)}\in[0,1]$
and that 
\[
\mathbb{E}[Y_{t}^{(k)}\mid X_{t}]=f^{(k)}(X_{t}).
\]
Here $f^{(k)}$ is the unknown mean reward function for the arm $k$.

Without the batch constraint, the game of nonparametric bandits plays
sequentially. At each step $t$, the statistician observes the context
$X_{t}$, and pulls an action $A_{t}\in\{1,-1\}$ according to a rule
$\pi_{t}:\mathcal{X}\mapsto\{1,-1\}$. Then she receives the corresponding
reward $Y_{t}^{(A_{t})}$. In this case, the rule $\pi_{t}$ for selecting
the action at time $t$ is allowed to depend on all the observations
strictly anterior to $t$.

In an $M$-batch game, the statistician needs to design  an $M$-batch policy $(\Gamma,\pi)$, where $\Gamma=\{t_{0,}t_{1,...,}t_{M}\}$
is a partition of the entire time horizon $T$ that satisfies $0=t_{0}<t_{1}<...<t_{M-1}<t_{M}=T$,
and $\pi=\{\pi_{t}\}_{t=1}^{T}$ is a sequence of random functions
$\pi_{t}:\mathcal{X}\mapsto\{1,-1\}$. The grid $\Gamma$ can be chosen adaptively, meaning that the statistician can use all information up to $t_{i-1}$ to determine $t_i$. More specifically, prior to the start of the game, she will specify the first batch $t_1$, and at the end of $t_1$, she will use all observations she have to decide the next batch $t_2$, and this process repeats in batches.
% is asked to divide the horizon
% $[1:T]$ into $M$ disjoint batches $[1:t_{1}]$, $[t_{1}+1:t_{2}]$,
% $\ldots,[t_{M-1}+1,T]$. 
In contrast to the case without the batch
constraint, only the rewards associated with timesteps prior to the
current batch are observed and available for making decisions for
the current batch. 
% More formally, an $M$-batch policy is composed
% of a pair $(\Gamma,\pi)$, where $\Gamma=\{t_{0,}t_{1,...,}t_{M}\}$
% is a partition of the entire time horizon $T$ that satisfies $0=t_{0}<t_{1}<...<t_{M-1}<t_{M}=T$,
% and $\pi=\{\pi_{t}\}_{t=1}^{T}$ is a sequence of random functions
% $\pi_{t}:\mathcal{X}\mapsto\{1,-1\}$. 
Specifically, let $\Gamma(t)$ be the batch
index for the time $t$, i.e., $\Gamma(t)$ is the unique integer
such that $t_{\Gamma(t)-1}<t\le t_{\Gamma(t)}$. Then at time $t$,
the available information for $\pi_{t}$ is only $\{X_{l}\}_{l=1}^{t}\cup\{Y_{l}^{(A_{l})}\}_{l=1}^{\Gamma(t)-1}$,
which we denote by $\mathcal{F}^{t}$. The statistician's policy $\pi_{t}$
at time $t$ is allowed to depend on $\mathcal{F}_{t}$. 

The goal of the statistician is to design an $M$-batch policy $(\Gamma,\pi)$
that can compete with an oracle that has perfect knowledge (i.e.,
the law of $(X_{t},Y_{t}^{(1)},Y_{t}^{(-1)})$) of the environment.
Formally, we define the expected cumulative regret as 
\begin{equation}
R_{T}(\pi)\coloneqq\mathbb{E}\left[\sum_{t=1}^{T}\left(f^{\star}(X_{t})-f^{(\pi_{t}(X_{t}))}(X_{t})\right)\right],\label{eq:regret}
\end{equation}
where $f^{\star}(x)\coloneqq\max_{k\in\{1,-1\}}f^{(k)}(x)$ is the
maximum mean reward one could obtain on the context $x$. Note here
we omit the dependence on $\Gamma$ for simplicity. 

\subsection{Assumptions}

We adopt two standard assumptions in the nonparametric bandits literature
\cite{rigollet2010nonparametric,perchet2013multi}. The first assumption
is on the smoothness of the mean reward functions.

\begin{assumption}[Smoothness]\label{assumption:smoothness}We assume
that the reward function for each arm is $(\beta,L)$-smooth, that
is, there exist $\beta\in(0,1]$ and $L>0$ such that for $k\in\{1,-1\}$,
\[
|f^{(k)}(x)-f^{(k)}(x')|\leq L\|x-x'\|_{2}^{\beta}
\]
holds for all $x,x'\in\mathcal{X}$. \end{assumption}

The second assumption is about the separation between the two reward
functions.

\begin{assumption}[Margin]\label{assumption:margin}We assume that
the reward functions satisfy the margin condition with parameter $\alpha>0$,
that is there exist $\delta_{0}\in(0,1)$ and $D_{0}>0$ such that
\[
\mathbb{P}_{X}\left(0<\left|f^{(1)}(X)-f^{(-1)}(X)\right|\leq\delta\right)\leq D_{0}\delta^{\alpha}
\]
holds for all $\delta\in[0,\delta_{0}]$. \end{assumption}

Assumption~\ref{assumption:margin} is related to the margin condition
in classification \cite{mammen1999smooth,tsybakov2004optimal,audibert2007fast}
and is introduced to bandits in \cite{Goldenshluger_2009,rigollet2010nonparametric,perchet2013multi}.
The margin parameter affects the complexity of the problem. Intuitively,
a small $\alpha$, say $\alpha\approx0$, means the two mean functions
are entangled with each other in many regions and hence it is challenging
to distinguish them; a large $\alpha$, on the other hand, means the
two reward functions are mostly well-separated.

From now on, we use $\mathcal{F}(\alpha,\beta)$ to denote the class
of nonparametric bandit instances (i.e., distributions over (\ref{eq:bandit}))
that satisfy Assumptions~\ref{assumption:smoothness}-\ref{assumption:margin}. 

\begin{remark} 
Throughout the paper, we assume that $\alpha\beta\leq1$.
By proposition 2.1 from \cite{rigollet2010nonparametric}, when $\alpha\beta>1$, the problem reduces to a static multi-armed bandit where one arm is always optimal, regardless of the context. Since our focus is on contextual bandits, we consider the case $\alpha\beta\le1$ hereafter.

% one of the arms will dominate the other one for the entire covariate
% space. The instance is reduced to a multi-armed bandit without covariates
% which is not the interest of the current paper. Therefore, we focus
% on the case $\alpha\beta\le1$ hereafter.

\end{remark}

\section{Fundamental limits of batched nonparametric bandits}

In this section, we establish minimax lower bounds for the regret achievable by any $M$-batch policy $(\Gamma,\pi)$; see Theorem~\ref{thm:lower-bound-adaptive}.
To begin with, we state a minimax lower bound, together with its proof, when the grid $\Gamma$ is prespecified, that is, the statistician divides the horizon
$[1:T]$ into $M$ disjoint batches $[1:t_{1}]$, $[t_{1}+1:t_{2}]$,
$\ldots,[t_{M-1}+1,T]$ before the game begins; see Theorem~\ref{thm:lower-bound}. As we will soon see, the proof of the lower bound with fixed grid is not only useful for establishing the lower bound for any general $M$-batch
policy $(\Gamma,\pi)$, but also instrumental in our development of an optimal policy to be detailed in Section~\ref{sec:algo}. 

% Somewhat unconventionally, we start with stating a minimax lower bound,
% as well as its proof, for regret minimization in batched nonparametric
% contextual bandits. As we will soon see, the proof of the lower bound
% is extremely instrumental in our development of an optimal $M$-batch
% policy $(\Gamma,\pi)$, to be detailed in Section~\ref{sec:algo}. 

Recall that $\mathcal{F}(\alpha,\beta)$ denotes the class of nonparametric
bandit instances (i.e., distributions over (\ref{eq:bandit})) that
obey Assumptions~\ref{assumption:smoothness}-\ref{assumption:margin}.
Compared to multi-armed bandits without covariates, the presence of contexts and the smooth function class no longer allows one to simply decompose the regret as the product between the expected number of times a suboptimal arm is pulled and the suboptimality gap. A more nuanced notion of expected number of suboptimal arm pulls for bandits with covariates was introduced in \cite{rigollet2010nonparametric}, and its relation to the regret was characterized via an elegant inequality. Nevertheless, a direct application of this relation to the batched setting would lead to suboptimal rates. To achieve a sharp result, we need to construct different families of hard instances for each batch and carefully lower bound the regret within it.
We have the following minimax lower bound for any $M$-batch policy with a fixed grid,
in which we define 
\[
\gamma\coloneqq\frac{\beta(1+\alpha)}{2\beta+d}\in(0,1).
\]

\begin{theorem}\label{thm:lower-bound}

Suppose that $\alpha\beta\le1$, and assume that $P_{X}$ is the uniform
distribution on $\mathcal{X}=[0,1]^{d}$. For any $M$-batch policy
$(\Gamma,\pi)$ where $\Gamma$ is prespecified, there exists a nonparametric bandit instance in $\mathcal{F}(\alpha,\beta)$
such that the regret of $(\Gamma,\pi)$ on this instance is lower
bounded by 
\[
R_{T}(\pi)
\ge\tilde{D}T^{\frac{1-\gamma}{1-\gamma^{M}}},
\]
where $\tilde{D}>0$ is a constant independent of $T$ and $M$.

\end{theorem}

\noindent Later, we will show that this lower bound is attained (modulo log factors) by our proposed algorithm Batched Successive Elimination with
Dynamic Binning (\myalg). Hence $T^{\frac{1-\gamma}{1-\gamma^{M}}}$ is the minimax optimal regret for the batched nonparametric bandit problem. 
See Section~\ref{subsec:lower-proof} for the proof of
this lower bound. 

\medskip

As a sanity check, one sees that as $M$ increases, the lower bound
    decreases. This is intuitive, as the policy is more powerful as $M$
    increases. As a result, the problem of batched nonparametric bandits
    becomes easier. When $M=\Omega(\log\log T)$, we recover the minimax regret $\Theta(T^{1-\gamma})$ in the fully online setting.

\subsection{Proof of Theorem~\ref{thm:lower-bound}\label{subsec:lower-proof}}

Let $(\Gamma,\pi)$ be the $M$-batch policy under consideration,
with 
\[
\Gamma=\{t_{0}=0,t_{1},t_{2},\ldots,t_{M}=T\}.
\]
Throughout this proof, we consider Bernoulli reward distributions,
that is $Y_{t}^{(1)},Y_{t}^{(-1)}$ are Bernoulli random variables
with mean $f^{(1)}(X_{t})$, and $f^{(-1)}(X_{t})$, respectively.
In addition, we fix $f^{(-1)}(x)=\frac{1}{2}$. Let $f$ be the mean
reward function of the first arm. To make the dependence on the reward
instance clear, we write the cumulative regret up to time $n$ as
$R_{n}(\pi;f)$. 

Our proof relies on a simple observation: the worst-case regret over
$[T]$ is larger than the worst-case regret over the first $i$ batches.
Formally, we have
\begin{equation}
\sup_{(f,\frac{1}{2})\in\mathcal{F}(\alpha,\beta)}R_{T}(\pi;f)\geq\max_{1\leq i\leq M}\sup_{(f,\frac{1}{2})\in\mathcal{F}(\alpha,\beta)}R_{t_{i}}(\pi;f).\label{eq:key-ineq}
\end{equation}
Though simple, this observation lends us freedom on choosing different
families of instances in $\mathcal{F}(\alpha,\beta)$ targeting different
batch indices $i$. 

Our proof consists of four steps. In Step 1, we reduce bounding the
regret of a policy to lower bounding its inferior sampling rate to
be defined. In Step 2, we detail the choice of different families
of instances for each different batch index $i$. Then in Step 3,
we apply an Assouad-type of argument to lower bound the average inferior
sampling rate of the family of hard instances. Lastly in Step 4, we
combine the arguments to complete the proof. 

\paragraph{Step 1: Relating regret to inferior sampling rate.}

Given an $M$-batch policy, we define its inferior sampling rate at
time $n$ on an instance $(f,\frac{1}{2})$ to be
\[
S_{n}(\pi;f)\coloneqq\mathbb{E}\left[\sum_{t=1}^{n}1\{\pi_{t}(X_{t})\neq\pi^{\star}(X_{t}),f(X_{t})\neq\frac{1}{2}\}\right].
\]
In words, $S_{n}(\pi;f)$ counts the number of times $\pi$ selects
the strictly suboptimal arm up to time $n$. Thanks to the following
lemma, we can reduce lower bounding the regret to the inferior sampling
rate. 

\begin{lemma}[Lemma 3.1 in~\cite{rigollet2010nonparametric}]Suppose
that $(f,\frac{1}{2})\in\mathcal{F}(\alpha,\beta)$. Then for any
$1\leq n\leq T$, we have 
\[
S_{n}(\pi;f)\leq Dn^{\frac{1}{1+\alpha}}R_{n}(\pi;f)^{\frac{\alpha}{1+\alpha}},
\]
for some constant $D>0$.\label{inferior-to-regret}\end{lemma}

As an immediate consequence of the above lemma, we obtain 
\begin{align*}
\sup_{(f,\frac{1}{2})\in\mathcal{F}(\alpha,\beta)}R_{T}(\pi;f) & \geq\max_{1\leq i\leq M}\sup_{(f,\frac{1}{2})\in\mathcal{F}(\alpha,\beta)}(\frac{1}{D})^{\frac{1+\alpha}{\alpha}}t_{i}^{-\frac{1}{\alpha}}(S_{t_{i}}(\pi;f))^{\frac{1+\alpha}{\alpha}}\\
 & =(\frac{1}{D})^{\frac{1+\alpha}{\alpha}}\max_{1\leq i\leq M}t_{i}^{-\frac{1}{\alpha}}\left[\sup_{(f,\frac{1}{2})\in\mathcal{F}(\alpha,\beta)}S_{t_{i}}(\pi;f)\right]^{\frac{1+\alpha}{\alpha}}.
\end{align*}
 From now on, we focus on lower bounding $\sup_{(f,\frac{1}{2})\in\mathcal{F}(\alpha,\beta)}S_{t_{i}}(\pi;f)$. 

\paragraph{Step 2: Introducing the family of reward instances for
$t_{i}$. }

Our construction of the family of hard instances is adapted from~\cite{rigollet2010nonparametric}.
Define $z_{1}=1$, and $z_{i}=\lceil t_{i-1}{}^{1/(2\beta+d)}\rceil$
for $i=2,3,\ldots,M$. Henceforth, we will fix some $i$ and write
$z_{i}$ as $z$. We partition $[0,1]^{d}$ into $z^{d}$ bins with
equal width. Denote the bins by $C_{j}$ for $j=1,...,z^{d}$, and
let $q_{j}$ be the center of $C_{j}$. 

Define a set of binary sequences $\Omega_{s}\coloneqq\{\pm1\}^{s}$,
with $s\coloneqq\lceil z^{d-\alpha\beta}\rceil$. For each $\omega\in\Omega_{s}$
we define a function $f_{\omega}:[0,1]^{d}\mapsto\mathbb{R}$:
\[
f_{\omega}(x)=\frac{1}{2}+\sum_{j=1}^{s}\omega_{j}\varphi_{j}(x),
\]
where $\varphi_{j}(x)=D_{\phi}z^{-\beta}\phi(2z(x-q_{j}))\mathbf{1}\{x\in C_{j}\}$
with $\phi(x)=(1-\|x\|_{\infty})^{\beta}\mathbf{1}\{\|x\|_{\infty}\le1\}$,
and $D_{\phi}=\min(2^{-\beta}L,1/4)$. In all, we consider the family
of reward instances
\begin{equation}\label{eq:instance_family}
    \mathcal{C}_{z}\coloneqq\left\{ f^{(1)}(x)=f_{\omega}(x),f^{(-1)}(x)=\frac{1}{2}\mid\omega\in\Omega_{s}\right\} .
\end{equation}
With slight abuse of notation, we also use $\mathcal{C}_{z}$ to denote
$\{f_{\omega}:\omega\in\Omega_{s}\}$. It is straightforward to check
that $\mathcal{C}_{z}\subseteq\mathcal{F}(\alpha,\beta).$

\paragraph{Step 3: Lower bounding the inferior sampling rate.}Fix
some $i\in[M]$, and consider $z=z_{i}$. Since $\mathcal{C}_{z}\subseteq\mathcal{F}(\alpha,\beta)$,
we have 
\[
\sup_{(f,\frac{1}{2})\in\mathcal{F}(\alpha,\beta)}S_{t_{i}}(\pi;f)\geq\sup_{f\in\mathcal{C}_{z}}S_{t_{i}}(\pi;f).
\]
Using the definitions of $\mathcal{C}_{z}$ and $S_{t_{i}}(\pi;f)$,
we have 

\begin{align*}
\sup_{f\in\mathcal{C}_{z}}S_{t_{i}}(\pi;f) & =\sup_{\omega\in\Omega_{s}}\mathbb{E}_{\pi,f_{\omega}}\left[\sum_{t=1}^{t_{i}}\mathbf{1}\{\pi_{t}(X_{t})\neq\textrm{sign}(f_{\omega}(X_{t})-\frac{1}{2}),f_{\omega}(X_{t})\neq\frac{1}{2}\}\right]\\
 & \ge\frac{1}{2^{s}}\sum_{\omega\in\Omega_{s}}\mathbb{E}_{\pi,f_{\omega}}\left[\sum_{t=1}^{t_{i}}\mathbf{1}\{\pi_{t}(X_{t})\neq\textrm{sign}(f_{\omega}(X_{t})-\frac{1}{2}),f_{\omega}(X_{t})\neq\frac{1}{2}\}\right].
\end{align*}
Since $f_{\omega}(x)=\frac{1}{2}$ for $x\notin\cup_{j=1,\ldots s}C_{j}$,
we further obtain

\begin{align}
\sup_{f\in\mathcal{C}_{z}}S_{t_{i}}(\pi;f) & \geq\frac{1}{2^{s}}\sum_{\omega\in\Omega_{s}}\sum_{t=1}^{t_{i}}\sum_{j=1}^{s}\mathbb{E}_{\pi,f_{\omega}}^{t}\left[\mathbf{1}\{\pi_{t}(X_{t})\neq\omega_{j},X_{t}\in C_{j}\}\right].\label{eq:avg-inferior}
\end{align}
Here we use $\mathbb{P}_{\pi,f_{\omega}}^{t}$ to denote the joint
distribution of $\{X_{l}\}_{l=1}^{t}\cup\{Y_{l}^{\pi_{l}(X_{l})}\}_{l=1}^{\Gamma(t)-1}$
, where $\Gamma(t)$ is the batch index for $t$, i.e., the unique
integer such that $t_{\Gamma(t)-1}<t\leq t_{\Gamma(t)}$. We use $\mathbb{E}_{\pi,f_{\omega}}^{t}$
to denote the corresponding expectation. Expand the right hand side
of~(\ref{eq:avg-inferior}) to see that 
\begin{equation}
\sup_{f\in\mathcal{C}_{z}}S_{t_{i}}(\pi;f)\geq\frac{1}{2^{s}}\sum_{j=1}^{s}\sum_{t=1}^{t_{i}}\sum_{\omega_{[-j]}\in\Omega_{s-1}}\underbrace{\sum_{h\in\{\pm1\}}\mathbb{E}_{\pi,f_{\omega_{[-j]}^{h}}}^{t}[\mathbf{1}\{\pi_{t}(X_{t})\neq h,X_{t}\in C_{j}\}]}_{W_{j,t,\omega_{[-j]}}},\label{eq:assouad}
\end{equation}
where $\omega_{[-j]}^{h}$ is the same as $\omega$ except for the
$j$-th entry being $h$. Note that here we use the fact that for
$f_{\omega_{[-j]}^{h}}$, the optimal arm in the bin $C_{j}$ is $h$.
We then relate $W_{j,t,\omega_{[-j]}}$ to a binary testing error,
\begin{align}
W_{j,t,\omega_{[-j]}} & =\frac{1}{z^{d}}\sum_{h\in\{\pm1\}}\mathbb{P}_{\pi,f_{\omega_{[-j]}^{h}}}^{t}(\pi_{t}(X_{t})\neq h\mid X_{t}\in C_{j})\nonumber \\
 & \ge\frac{1}{4z^{d}}\exp\left[-\mathrm{KL}(\mathbb{P}_{\pi,f_{\omega_{[-j]}^{-1}}}^{t},\mathbb{P}_{\pi,f_{\omega_{[-j]}^{1}}}^{t})\right],\label{eq:le-cam}
\end{align}
where the second step invokes Le Cam's method. Under the batch setting,
at time $t$, the available information is only up to $t_{\Gamma(t)-1}$.
Consequently, we can apply Lemma~\ref{lemma:kl-neg-pos} to obtain

\begin{align}\label{eq:kl-bound}
\KL(\mathbb{P}_{\pi,f_{\omega_{[-j]}^{-1}}}^{t},\mathbb{P}_{\pi,f_{\omega_{[-j]}^{1}}}^{t})
=\KL(\mathbb{P}_{\pi,f_{\omega_{[-j]}^{-1}}}^{t_{\Gamma(t)-1}},\mathbb{P}_{\pi,f_{\omega_{[-j]}^{1}}}^{t_{\Gamma(t)-1}})\le 2z^{-(2\beta+d)}t_{\Gamma(t)-1}.
\end{align}
Combining (\ref{eq:assouad}), (\ref{eq:le-cam}),
and (\ref{eq:kl-bound}), we arrive at
\begin{align*}
\sup_{f\in\mathcal{C}_{z}}S_{t_{i}}(\pi;f) & \ge\frac{1}{8}\sum_{j=1}^{s}\sum_{t=1}^{t_{i}}\frac{1}{z^{d}}\exp\left(-2z^{-(2\beta+d)}t_{\Gamma(t)-1}\right)\\
 & \geq\frac{1}{8}\sum_{j=1}^{z^{d-\alpha\beta}}\sum_{l=1}^{i}\frac{t_{l}-t_{l-1}}{z^{d}}\exp\left(-2z^{-(2\beta+d)}t_{l-1}\right)\\
 & \geq\frac{1}{8}\sum_{j=1}^{z^{d-\alpha\beta}}\sum_{l=1}^{i}\frac{t_{l}-t_{l-1}}{z^{d}}\exp\left(-2z^{-(2\beta+d)}t_{i-1}\right),
\end{align*}
where the second line uses the fact that $s=\lceil z^{d-\alpha\beta}\rceil$,
and the last inequality holds since $t_{l-1}\le t_{i-1}$ for all
$1\le l\le i$. Now recall that $z=z_{i}=\lceil(t_{i-1}){}^{1/(2\beta+d)}\rceil$
for $i\geq1$, and $z=1$ for $i=1$. We can continue the lower bound
to see that 
\begin{align*}
\sup_{f\in\mathcal{C}_{z_{i}}}S_{t_{i}}(\pi;f) & \ge\frac{1}{8}\sum_{j=1}^{z^{d-\alpha\beta}}\sum_{l=1}^{i}\frac{t_{l}-t_{l-1}}{z^{d}}\exp\left(-2z^{-(2\beta+d)}t_{i-1}\right)\\
 & \ge c^{\star}\sum_{j=1}^{z^{d-\alpha\beta}}\sum_{l=1}^{i}\frac{t_{l}-t_{l-1}}{z^{d}}\\
 & =c^{\star}\cdot\frac{t_{i}}{z^{\alpha\beta}}=\begin{cases}
c^{\star}\cdot\frac{t_{i}}{t_{i-1}^{\frac{\alpha\beta}{2\beta+d}}}, & i>1\\
c^{\star}t_{1}, & i=1
\end{cases},
\end{align*}
 for some $c^{\star}>0$. 

\paragraph{Step 4: Combining bounds together.}Combining the previous
arguments together leads to the conclusion that 
\begin{align}
\sup_{(f,\frac{1}{2})\in\mathcal{F}(\alpha,\beta)}R_{T}(\pi;f) & \geq\max_{1\leq i\leq M}\sup_{f\in\mathcal{C}_{z_{i}}}R_{t_{i}}(\pi;f)\nonumber \\
 & \geq(\frac{1}{D})^{\frac{1+\alpha}{\alpha}}\max_{1\leq i\leq M}t_{i}^{-\frac{1}{\alpha}}\left[\sup_{f\in\mathcal{C}_{z_{i}}}S_{t_{i}}(\pi;f)\right]^{\frac{1+\alpha}{\alpha}}\nonumber \\
 & \gtrsim\max\left\{ t_{1},\frac{t_{2}}{t_{1}^{\gamma}},...,\frac{T}{t_{M-1}^{\gamma}}\right\} \label{eq:lower-bound-final}\\
 & \geq\tilde{D}T^{\frac{1-\gamma}{1-\gamma^{M}}}.\nonumber 
\end{align}
This finishes the proof.

\subsection{Lower bound for general $M$-batch policy}
% \paragraph{Does adaptivity help.}
% One natural question to ask is if the statistician can choose the
% batch size adaptively instead of picking the grid $\Gamma$ before the game starts, would she achieve a fundamentally better regret? More precisely, for an adaptive grid, the statistician can use all information up to $t_{i-1}$ to determine $t_i$. Prior to the game, she will specify the first batch $t_1$, and at the end of $t_1$, she will use all observations she have to decide the next batch $t_2$, and this process repeats in batches. 
Now we are ready to state the minimax lower bound for any general $M$-batch
policy $(\Gamma,\pi)$, i.e., when the grid $\Gamma$ is allowed to be adaptively chosen.
% The following
% theorem shows that the benefit of an adaptive grid is minimal if one ignores logarithmic factors.

\begin{theorem}\label{thm:lower-bound-adaptive}

Suppose that $\alpha\beta\le1$, and assume that $P_{X}$ is the uniform
distribution on $\mathcal{X}=[0,1]^{d}$. For any $M$-batch policy
$(\Gamma,\pi)$, there exists a nonparametric bandit
instance in $\mathcal{F}(\alpha,\beta)$ such that the regret of $\pi$
on this instance is lower bounded by 
\[
R_{T}(\pi)
\ge\tilde{D}_{1}(\frac{1}{M})^{\tilde{D}_2}\cdot T^{\frac{1-\gamma}{1-\gamma^{M}}},
\]
where $\tilde{D}_{1},\tilde{D}_{2}>0$ are constants independent of
$T$ and $M$.

\end{theorem}

\noindent See Appendix~\ref{sec:adaptive-lower-proof} for the proof. 

Since our focus is on $M\apprle\log\log T$ (when $M\apprge\log\log T$,
by Corollary~\ref{coro:loglogt} there exists an algorithm whose regret attains the optimal
fully online regret), we can see Theorem~\ref{thm:lower-bound} and Theorem~\ref{thm:lower-bound-adaptive} differ at most
by poly-log factors in $T$.

Unlike the fixed grid case where we choose a specific family of hard
instances to target the regret in a certain batch, we cannot directly do
so when the grid is adaptively selected because the adversary does
not know $\{t_{i}\}_{i=1}^{M}$ in advance. Inspired by \cite{zijun2020batch}, we overcome this difficulty by using an appropriately defined
bad event that happens with sufficient probability to reduce the adaptive
case to the fixed grid case. A significant technical difference to MAB lies in the construction of a family of nonparametric bandit instances that are simultaneously hard for all batches. The full proof can be found in Appendix~\ref{sec:adaptive-lower-proof}.

\subsection{Implications on design of optimal $M$-batch policy \label{subsec:implication-on-upper-bound}}

As we have mentioned, the proof of the lower bound with fixed grid, i.e., Theorem~\ref{thm:lower-bound}
facilitates the design of optimal $M$-batch policy. 

\paragraph{Grid selection.}First, the lower bound of the whole horizon
is reduced to the worst-case regret over a specific batch; see~(\ref{eq:key-ineq}).
Consequently, we need to design the grid $\Gamma=(t_{0},t_{1},t_{2},\ldots,t_{M-1},t_{M})$
such that the total regret is evenly distributed across batches. 
More concretely, in view of the lower bound~\eqref{eq:lower-bound-final}, one needs to set $t_1 \asymp \frac{t_i}{t_{i-1}^\gamma} \asymp T^{\frac{1-\gamma}{1-\gamma^M}}$ for $2\leq i \leq M$.

\paragraph{Dynamic binning.} In addition, in the proof of the lower
bound, for each different batch $i$, we use different families of
hard reward instances, parameterized by the number of bins $z_{i}=\lceil t_{i-1}^{1/(2\beta+d)}\rceil$.
In other words, from the lower bound perspective, the granularity
(i.e., the bin width $1/z_{i}$) at which we investigate the mean
reward function depends crucially on the grid points $\{t_{i}\}$:
the larger the grid point $t_{i}$, the finer the granularity. This
key observation motivates us to consider the batched successive elimination
with dynamic binning algorithm to be introduced below. 

\section{Batched successive elimination with dynamic binning \label{sec:algo}}

\begin{algorithm}[t]
\caption{Batched successive elimination with dynamic binning (\myalg)}
\label{alg:adaptive-bin}
\begin{lyxcode}
\textbf{$\mathbf{Input}$}:~Batch~size~$M$,~grid~$\Gamma=\{t_{i}\}_{i=0}^{M}$,~split~factors~$\{g_{i}\}_{i=0}^{M-1}$.

$\mathcal{L}\leftarrow\mathcal{B}_{1}$

\textbf{$\mathbf{for}$~$C\in\mathcal{L}$~$\mathbf{do}$}
\begin{lyxcode}
$\mathcal{I}_{C}=\mathcal{I}$
\end{lyxcode}
\textbf{$\mathbf{for}$}~$i=1,...,M-1$~\textbf{$\mathbf{do}$}
\begin{lyxcode}
\textbf{$\mathbf{for}$}~$t=t_{i-1}+1,...,t_{i}$~\textbf{$\mathbf{do}$}
\begin{lyxcode}
$C\leftarrow\mathcal{L}(X_{t})$

Pull~an~arm~from~$\mathcal{I}_{C}$~in~a~round-robin~way.

\textbf{$\mathbf{if}$}~$t=t_{i}$~\textbf{$\mathbf{then}$}
\begin{lyxcode}
Update~$\mathcal{L}$~and~$\{\mathcal{I}_{C}\}_{C\in\mathcal{L}}$~by~Algorithm~\textrm{\ref{algo-subroutine}}~$(\mathcal{L},\{\mathcal{I}_{C}\}_{C\in\mathcal{L}},i,g_{i})$.
\end{lyxcode}
\end{lyxcode}
\end{lyxcode}
$\mathbf{for}$~$t=t_{M-1}+1,...,T$~$\mathbf{do}$
\begin{lyxcode}
$C\leftarrow\mathcal{L}(X_{t})$

Pull~any~arm~from~$\mathcal{I}_{C}$.
\end{lyxcode}
\end{lyxcode}
\end{algorithm}

\begin{algorithm}
\caption{Tree growing subroutine}
\label{algo-subroutine}
\begin{lyxcode}
$\mathbf{Input}$:~Active~nodes~$\mathcal{L}$,~active~arm~sets~$\{\mathcal{I}_{C}\}_{C\in\mathcal{L}}$,~batch~number~$i$,~split~factor~$g_{i}$.

$\mathcal{L}^{\prime}\leftarrow\{\}$

$\mathbf{for}$~\textbf{$C\in\mathcal{L}$}~$\mathbf{do}$
\begin{lyxcode}
$\mathbf{if}$~$|\mathcal{I}_{C}|=1$~$\mathbf{then}$
\begin{lyxcode}
$\mathcal{L^{\prime}}\leftarrow\mathcal{L^{\prime}}\cup\{C\}$

Proceed~to~next~$C$~in~the~iteration.
\end{lyxcode}
$\bar{Y}_{C,i}^{\max}\leftarrow\max_{k\in\mathcal{I}_{C}}\bar{Y}_{C,i}^{(k)}$

$\mathbf{for}$~$k\in\mathcal{I}_{C}$~$\mathbf{do}$
\begin{lyxcode}
$\mathbf{if}$~$\bar{Y}_{C,i}^{\max}-\bar{Y}_{C,i}^{(k)}>U(m_{C,i},T,C)$~$\mathbf{then}$~$\mathcal{I}_{C}\leftarrow\mathcal{I}_{C}-\{k\}$
\end{lyxcode}
$\mathbf{if}$~$|\mathcal{I}_{C}|>1$~$\mathbf{then}$
\begin{lyxcode}
$\mathcal{I}_{C^{\prime}}\leftarrow\mathcal{I}_{C}$~$\mathbf{for}$~$C^{\prime}\in\textrm{child}(C,g_{i})$

$\mathcal{L^{\prime}}\leftarrow\mathcal{L^{\prime}}\cup\textrm{child}(C,g_{i})$
\end{lyxcode}
$\mathbf{else}$
\begin{lyxcode}
$\mathcal{L^{\prime}}\leftarrow\mathcal{L^{\prime}}\cup\{C\}$
\end{lyxcode}
\end{lyxcode}
Return~$\mathcal{L^{\prime}}$
\end{lyxcode}
\end{algorithm}
In this section, we present the batched successive elimination with
dynamic binning policy (\myalg) that nearly
attains the minimax lower bound, up to log factors; see Algorithm~\ref{alg:adaptive-bin}. On a high level,
Algorithm \ref{alg:adaptive-bin} gradually partitions the covariate
space $\mathcal{X}$ into smaller hypercubes (i.e., bins) throughout
the batches based on a list of carefully chosen cube widths, and reduces
the nonparametric bandit in each cube to a bandit problem without
covariates. 

\begin{figure}[t]
\begin{center}
\begin{tikzpicture}[level distance=3cm, sibling distance=2cm, every node/.style={circle, draw, align=center, minimum size=1.4cm}]   
\node [fill=red!30, text=black, circle, draw] {$[0,1]$}     
	child {node [fill=red!30, text=black, circle, draw] {$[0,\frac{1}{4})$}       
		child {node [fill=green!30, text=black, circle, draw] {$[0,\frac{1}{12})$}}       
		child {node [fill=green!30, text=black, circle, draw] {$[\frac{1}{12},\frac{1}{6})$}} 
		child {node [fill=green!30, text=black, circle, draw] {$[\frac{1}{6},\frac{1}{4})$}}   
		}    
	child {node [fill=green!30, text=black, circle, draw] {$[\frac{1}{4},\frac{1}{2})$}}
	child {node [fill=green!30, text=black, circle, draw] {$[\frac{1}{2},\frac{3}{4})$}}
	child {node [fill=red!30, text=black, circle, draw] {$[\frac{3}{4},1]$}       
		child {node [fill=green!30, text=black, circle, draw] {$[\frac{3}{4},\frac{5}{6})$}}       
		child {node [fill=green!30, text=black, circle, draw] {$[\frac{5}{6},\frac{11}{12})$}} 
		child {node [fill=green!30, text=black, circle, draw] {$[\frac{11}{12},1]$}}   
		}; 
\end{tikzpicture}
\end{center}

\caption{An example of the tree growing process for $d=1, M=3, G=\{4,3,1\}$. The root node is at depth 0. For the first batch, the 4 nodes located at depth 1 of the tree were used. Both $[\frac{1}{4},\frac{1}{2})$ and $[\frac{1}{2},\frac{3}{4})$ only had one active arm remaining so they were not further split and remained in the set of active nodes (green). Meanwhile, $|\mathcal{I}_{[0,\frac{1}{4})}|=|\mathcal{I}_{[\frac{3}{4},1]}|=2$ so each of them was split into 3 smaller nodes, and both nodes were marked as inactive (red). For the second batch, all the green nodes were actively used but arm elimination was performed at the end of batch 2 only for nodes located at depth 2 (the green nodes at depth 1 already have 1 active arm remaining so there is no need to eliminate again).}
\label{tree-example}
\end{figure}
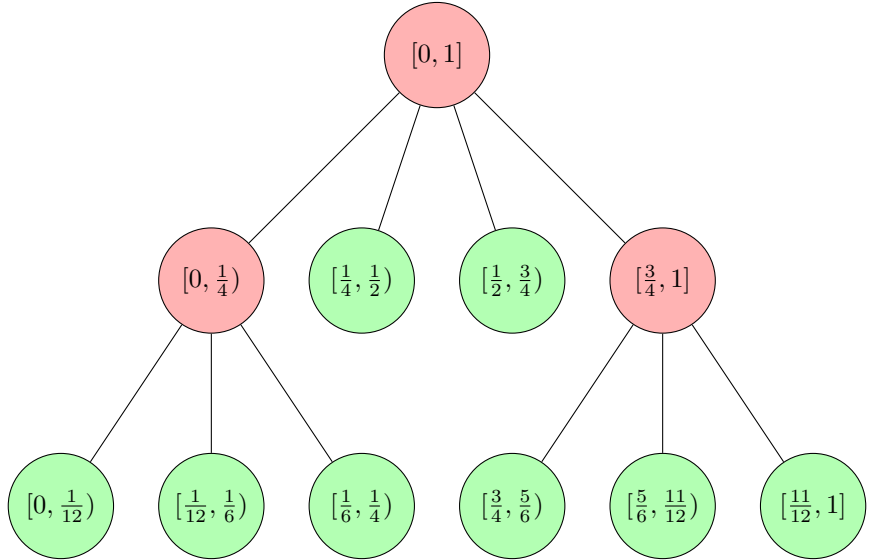

\paragraph{A tree-based interpretation. }The process is best illustrated
with the notion of a tree $\mathcal{T}$ of depth $M$; see Figure~\ref{tree-example}.
Each layer of of the tree $\mathcal{T}$ is a set of bins that form
a regular partition of $\mathcal{X}$ using hypercubes with equal
widths. And the common width of the bins $\mathcal{B}_{i}$ in layer
$i$ is dictated by a list $\{g_{i}\}_{i=0}^{M-1}$ of split factors.
More precisely, we let 
\begin{equation}
w_{i}\coloneqq(\prod_{l=0}^{i-1}g_{l})^{-1}\label{eq:width}
\end{equation}
be the width of the cubes in the $i$-th layer $\mathcal{B}_{i}$ for $i\ge1$, and $w_0=1$.
In other words, $\mathcal{B}_{i}$ contains all the cubes 

\[
C_{i,\bm{v}}=\{x\in\mathcal{X}:(v_{j}-1)w_{i}\le x_{j}<v_{j}w_{i},1\le j\le d\},
\]
where $\bm{v}=(v_{1},v_{2},\ldots,v_{d})\in[\frac{1}{w_{i}}]^{d}$.
As a result, there are in total $(\frac{1}{w_{i}})^{d}$ bins in $\mathcal{B}_{i}$. 

Algorithm~\ref{alg:adaptive-bin} proceeds in batches and maintains
two key objects: (1) a list $\mathcal{L}$ of active bins, and (2)
the corresponding active arms $\mathcal{I}_{C}$ for each $C\in\mathcal{L}$;
see Figure~\ref{tree-example} for an example. Specifically, prior
to the game (i.e., prior to the first batch), $\mathcal{L}$ is set
to be $\mathcal{B}_{1}$, all bins in layer 1, and $\mathcal{I}_{C}=\{1,-1\}$
for all $C\in\mathcal{L}$. Within this batch, the statistician tries
the arms in $\mathcal{I}_{C}$ equally likely for all bins in $\mathcal{L}$.
Then at the end of the batch, given the revealed rewards in this batch,
we update the active arms $\mathcal{I}_{C}$ for each $C\in\mathcal{L}$
via successive elimination. If no arm were eliminated from $\mathcal{I}_{C}$,
this suggests that the current bin is not fine enough for the statistician
to tell the difference between the two arms. As a result, she splits
the bin $C\in\mathcal{L}$ into its children $\textrm{child}(C)$
in $\mathcal{T}$. All the child nodes will be included in $\mathcal{L}$,
while the parent $C$ stops being active (i.e., $C$ is removed from
$\mathcal{L}$). The whole process repeats in a batch fashion. \footnote{For the final batch $M$, the split factor $g_{M-1}=1$ by default because there is no need to further partition the nodes for estimation.}

\paragraph{Grid $\Gamma$ and split factors $\{g_{i}\}_{i=0}^{M-1}$.}
As one can see, the split factor $g_{i}$ controls how many children
a node at layer $i$ can have and its appropriate choice is crucial
for obtaining small regret. Intuitively, $g_{i}$ should be selected
in a way such that a node $C_{i+1}$ with width $w_{i}$ can fully
leverage the number of samples allocated to it during the $(i+1)$-th
batch. With these goals in mind, we design the grid $\Gamma=\{t_{i}\}$
and split factors $\{g_{i}\}$ as follows. Recall that $\gamma=\frac{\beta(1+\alpha)}{2\beta+d}$.
We set 
\[
b=\Theta\left(T^{\frac{1-\gamma}{1-\gamma^{M}}}\right).
\]
The split factors are chosen according to 
\begin{align}
g_{0}=\lfloor b^{\frac{1}{2\beta+d}}\rfloor,\qquad\text{and}\qquad g_{i}=\lfloor g_{i-1}^{\gamma}\rfloor,i=1,...,M-2. \label{eq:split-factors}
\end{align}
In addition, the grid is chosen such that
\begin{align}
t_{i}-t_{i-1} & =\lfloor l_{i}w_{i}^{-(2\beta+d)}\log(Tw_{i}^{d})\rfloor,1\le i\le M-1,\label{eq:batch-bin-size}
\end{align}
where $l_{i}>0$ is a constant to be specified later. It is easy to
check that with these choices, we have 
\begin{equation*}
t_{1}\asymp T^{\frac{1-\gamma}{1-\gamma^{M}}},\qquad\text{and}\qquad t_{i}=\lfloor b(t_{i-1})^{\gamma}\rfloor,\quad\text{for }i=2,...,M.
\end{equation*}
In particular, we set $b$ properly to make $t_{M}=T$. Indeed, these
choices taken together meet the expectation laid out in Section~\ref{subsec:implication-on-upper-bound}:
we need to choose the grid and the split factors appropriately so
that (1) the total regret spreads out across different batches, and
(2) the granularity becomes finer as we move further to later batches. 

\paragraph{When to eliminate arms?} Now we zoom in on the elimination
process described in Algorithm~\ref{algo-subroutine}. The basic
idea follows from successive elimination in the bandit literature
\cite{even2006action,perchet2013multi,zijun2020batch}: the statistician
eliminates an arm from $\mathcal{I}_{C}$ if she expects the arm to
be suboptimal in the bin $C$ given the rewards collected in $C$.
Specifically, for any node $C\in\mathcal{T}$, define 
\[
U(\tau,T,C)\coloneqq4\sqrt{\frac{\log(2T|C|^{d})}{\tau}},
\]
where $|C|$ denotes the width of the bin. Let $m_{C,i}\coloneqq\sum_{t=t_{i-1}+1}^{t_{i}}\mathbf{1}\{X_{t}\in C\}$
be the number of times we observe contexts from $C$ in batch $i$.
We then define for $k\in\{1,-1\}$ that 
\[
\bar{Y}_{C,i}^{(k)}\coloneqq\frac{\sum_{t=t_{i-1}+1}^{t_{i}}Y_{t}\cdot\mathbf{1}\{X_{t}\in C,A_{t}=k\}}{\sum_{t=t_{i-1}+1}^{t_{i}}\mathbf{1}\{X_{t}\in C,A_{t}=k\}},
\]
which is the empirical mean reward of arm $k$ in node $C$ during
the $i$-th batch. It is easy to check that $\bar{Y}_{C,i}^{(k)}$
has expectation $\bar{f}_{C}^{(k)}$ given by 
\[
\bar{f}_{C}^{(k)}\coloneqq\mathbb{E}[f^{(k)}(X)\mid X\in C]=\frac{1}{\mathbb{P}_{X}(C)}\int_{C}f^{(k)}(x)\mathrm{d}\mathbb{P}_{X}(x).
\]
Similarly, we define the average optimal reward in bin $C$ to be
\[
\bar{f}_{C}^{\star}\coloneqq\frac{1}{\mathbb{P}_{X}(C)}\int_{C}f^{\star}(x)\mathrm{d}\mathbb{P}_{X}(x).
\]

The elimination threshold $U(m_{C,i},T,C)$ is chosen such that an
arm $k$ with $\bar{f}_{C}^{\star}-\bar{f}_{C}^{(k)}\gg|C|^{\beta}$
is eliminated with high probability at the end of batch $i$. Therefore,
when $|\mathcal{I}_{C}|>1$, the remaining arms are statistically
indistinguishable from each other, so $C$ is split into smaller
nodes to estimate those arms more accurately using samples from future
batches. On the other hand, when $|\mathcal{I}_{C}|=1$, the remaining
arm is optimal in $C$ with high probability---a consequence of the
smoothness condition, and it will be exploited in the later batches.

\paragraph{Connections and differences with ABSE in~\cite{perchet2013multi}.}
In appearance, \myalg (Algorithm~\ref{alg:adaptive-bin}) looks quite
similar to the Adaptively Binned Successive Elimination (ABSE) proposed
in~\cite{perchet2013multi}. However, we would like to emphasize
several fundamental differences. First, the motivations for the algorithms
are completely different. \cite{perchet2013multi} designs ABSE to
adapt to the unknown margin condition $\alpha$, while our focus is
to tackle the batch constraint. In fact, without the batch
constraints, if $\alpha$ is known, adaptive binning is not needed
to achieve the optimal regret \cite{perchet2013multi}. This is certainly
not the case in the batched setting. Fixing the number of bins used
across different batches is suboptimal because one can construct instances
that cause the regret incurred during a certain batch to explode.
We will expand on this phenomenon in Section~\ref{subsec:static-failure}.
Secondly, the algorithm in \cite{perchet2013multi} partitions a bin
into a \emph{fixed} number $2^{d}$ of smaller ones once the original
bin is unable to distinguish the remaining arms. In this way, the
algorithm can adapt to the difference in the local difficulty of the
problem. In comparison, one of our main contributions is to carefully
design the list of \emph{varying} split factors that allows the new
cubes to maximally utilize the number of samples allocated to it during
the next batch. 

\subsection{Regret guarantees}

Now we are ready to present the regret performance of \myalg
(Algorithm~\ref{alg:adaptive-bin}).

\begin{theorem}\label{thm:upper-bound}

Suppose that $\alpha\beta\le1$. Fix any constant $D_{1}>0$ and suppose
that $M\le D_{1}\log T$. Equipped with the grid and split factors
list that satisfy (\ref{eq:batch-bin-size}) and (\ref{eq:split-factors}),
the policy $\hat{\pi}$ given by Algorithm~\ref{alg:adaptive-bin}
obeys
\[
R_{T}(\hat{\pi})
\le\tilde{C}(\log T)^{2}\cdot T^{\frac{1-\gamma}{1-\gamma^{M}}},
\]
where $\tilde{C}>0$ is a constant independent of $T$ and $M$.
\end{theorem}

\noindent See Appendix~\ref{sec:upper-proof} for the proof. 

\medskip

While Theorem~\ref{thm:upper-bound} requires $M\lesssim\log T$,
we see from the corollary below that it is in fact sufficient to show
the optimality of Algorithm~\ref{alg:adaptive-bin}. 

\begin{corollary}\label{coro:loglogt}

As long as $M\geq D_{2}\log\log(T)$, where $D_{2}$ depends on $\gamma=\frac{\beta(1+\alpha)}{2\beta+d}$,
Algorithm~\ref{alg:adaptive-bin} achieves
\[
R_{T}(\hat{\pi})
\le\tilde{C}(\log T)^{2}\cdot T^{1-\gamma},
\]
where $\tilde{C}>0$ is a constant independent of $T$ and $M$.

\end{corollary}

Theorem~\ref{thm:upper-bound}, together with Corollary~\ref{coro:loglogt}
and Theorem~\ref{thm:lower-bound-adaptive} establish the fundamental limits
of batch learning for the nonparametric bandits with covariates, as
well as the optimality of \myalg, up to logarithmic factors. To see
this, when $M\lesssim\log\log(T)$, the upper bound in Theorem \ref{thm:upper-bound}
matches the lower bounds in Theorem~\ref{thm:lower-bound} and Theorem~\ref{thm:lower-bound-adaptive}, apart
from log factors. On the other end, when $M\gtrsim\log\log(T)$, Algorithm~\ref{alg:adaptive-bin},
while splitting the horizon into $M$ batches, achieves the optimal
regret (up to log factors) for the setting without the batch constraint \cite{perchet2013multi}.
It is evident that Algorithm~\ref{alg:adaptive-bin} is optimal in
this case. 

\subsection{Numerical experiments}

In this section, we provide some experiments on the empirical performance
of Algorithm~\ref{alg:adaptive-bin}. We set $T=50000,d=\beta=1,\alpha=0.2$.
We let $P_{X}$ be the uniform distribution on $[0,1]$. Denote $q_{j}=(j-1/2)/4$
and $C_{j}=[q_{j}-1/8,q_{j}+1/8]$ for $1\le j\le4$. For the mean
reward functions, we choose $f^{(1)},f^{(-1)}:[0,1]\rightarrow\mathbb{R}$
such that
\[
f^{(1)}(x)=\frac{1}{2}+\sum_{j=1}^{4}\omega_{j}\varphi_{j}(x),\qquad f^{(-1)}(x)=\frac{1}{2},
\]
where $\omega_{j}'s$ are sampled i.i.d.~from $\mathrm{Rad}(\frac{1}{2})$,
$\varphi_{j}(x)=\frac{1}{4}\phi(8(x-q_{j}))\mathbf{1}\{x\in C_{j}\}$
and $\phi(x)=(1-|x|)\mathbf{1}\{|x|\le1\}$. We let $Y^{(k)}\sim\mathrm{Bernoulli}(f^{(k)}(x))$.
To illustrate the performance of Algorithm~\ref{alg:adaptive-bin},
we compare it with the Binned Successive Elimination (BSE) policy
from \cite{perchet2013multi}, which is shown to be minimax optimal
in the fully online case. Figure~\ref{fig:exp} shows the regret of
Algorithm~\ref{alg:adaptive-bin} under different batch budegts.
One can see that it is sufficient to have $M=5$ batches to achieve
the fully online efficiency.

\begin{figure}         
\centering         
\includegraphics[scale=0.4]{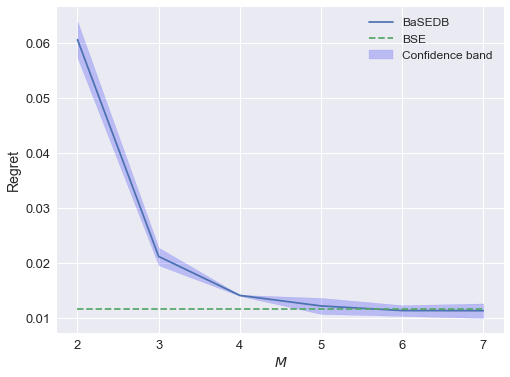}         
\caption{Regret vs. batch budget $M$.}         
\label{fig:exp}     
\end{figure}

\subsection{Failure of static binning\label{subsec:static-failure}}

We have seen the power of dynamic binning in solving batched nonparametric
bandits by establishing its rate-optimality in minimizing regret.
Now we turn to a complimentary but intriguing question: is it necessary
to use dynamic binning to achieve optimal regret under the batch constraint?
To formally address this question, we investigate the performance
of successive elimination with \emph{static} binning, i.e., Algorithm~\ref{alg:adaptive-bin}
with $g_{0}=g$, and $g_{1}=g_{2}=\cdots g_{M-2}=1$. Although static
binning works when $M$ is large (e.g., a single choice of $g$ attains
the optimal regret~\cite{rigollet2010nonparametric,perchet2013multi}
in the fully online setting), we show that it must fail when $M$
is small. 

To bring the failure mode of static binning into focus, we consider
the simplest scenario when $M=3$, and $\alpha=\beta=d=1$. Note
that the successive elimination with \emph{static} binning algorithm
is parameterized by the grid choice $\Gamma=\{t_{0}=0,t_{1},t_{2},t_{3}=T\}$
and the fixed number $g$ of bins. The following theorem formalizes
the failure of static binning in achieving optimal regret when $M=3$.

\begin{theorem}\label{thm:static-lower}

Consider $M=3$, and $\alpha=\beta=d=1$. For any choice of $1\leq t_{1}<t_{2}\leq T-1$,
and any choice of $g$, there exists a nonparametric bandit instance
in $\mathcal{F}(1,1)$ such that the resulting successive elimination
with \emph{static} binning algorithm $\hat{\pi}_{\mathrm{static}}$
satisfies 

\[
R_{T}(\hat{\pi}_{\mathrm{static}})
\geq\tilde{C}_{1}T^{\frac{9}{19}+\kappa},
\]
for some $\kappa, \tilde{C}_{1}>0$ that are independent
of $T$. Here $T^{\frac{9}{19}}$ is the optimal regret achieved by
\myalg---an successive elimination algorithm with dynamic binning.
\end{theorem}

\noindent While the formal proof is deferred to Appendix~\ref{sec:Proof-of-Theorem-failure},
we would like to immediately point out the intuition underlying the
failure of static binning. 

\paragraph{Necessary choice of grid $\Gamma$.} It is evident from
the proof of the minimax lower bound (Theorem~\ref{thm:lower-bound})
that one needs to set $t_{1}\asymp T^{9/19}$, and $t_{2}\asymp T^{15/19}$.
Otherwise, the inequality~(\ref{eq:lower-bound-final}) guarantees
the worst-case regret of $\hat{\pi}_{\mathrm{static}}$ exceeds the
optimal one $T^{\frac{9}{19}}$. Consequently, we can focus on the
algorithm with $t_{1}\asymp T^{9/19}$, $t_{2}\asymp T^{15/19}$,
and only consider the design choice $g$. 

\paragraph{Why fixed $g$ fails.} As a baseline for comparison,
recall that in the optimal algorithm with dynamic binning, we set
$g_{0}\asymp T^{3/19}$, and $g_{0}g_{1}\asymp T^{5/19}$ so that
the worst case regret in three batches are all on the order of $T^{\frac{9}{19}}$.
In view of this, we split the choice of $g$ into three cases. 
\begin{itemize}
\item Suppose that $g\gg T^{3/19}$. In this case, we can construct an instance
such that the reward difference only appears on an interval with length
$1/z\gg1/g$; see Figure~\ref{fig:N_ge_g}. In other words, the static
binning is finer than that in the reward instance. As a result, the
number of pulls in the smaller bin (used by the algorithm) in the
first batch is not sufficient to tell the two arms apart, that is
with constant probability, arm elimination will not happen after the
first batch. This necessarily yields the blowup of the regret in the
second batch. 
\item Suppose that $g\ll T^{3/19}$. In this case, we can construct an instance
such that the reward difference only appears on an interval with length
$1/z\ll1/g$; see Figure~\ref{fig:N_ll_g}. In other words, the static
binning is coarser than that in the reward instance. Since the aggregated
reward difference on the larger bin is so small, the number of pulls
in the larger bin (used by the algorithm) in the first batch is still
not sufficient to result in successful arm elimination. Again, the
regret on the second batch blows up. 
\item Suppose that $g\asymp T^{3/19}$. Since this choices matches $g_{0}$
used in the optimal dynamic binning algorithm, there is no reward
instance that can blow up the regret in the first two batches. Nevertheless,
since $g\ll g_{0}g_{1}\asymp T^{5/19}$, one can construct the instance
similar to the previous case (i.e., Figure~\ref{fig:N_ll_g}) such
that the regret on the third batch blows up. 
\end{itemize}
\begin{figure}     
\centering    

\begin{tikzpicture}[scale=1.2] % Adjust scale for overall size   
% Draw the horizontal line for the axis       
\draw[thick, -latex] (0,0) -- (12,0) 
node[anchor=north west] {$x$};    
       
% Draw ticks for intervals on the axis       
\foreach \x in {1, 2, ..., 11} 
{         \draw[blue!50, thick] (\x,0.1) -- (\x,-0.1);       }   
          
% Draw the edges of the triangle red, except for the base       
\draw[orange, thick] (0,0) -- (2,2.5);       
\draw[orange, thick] (2,2.5) -- (4,0);        
   
% Add an upper brace for the triangle interval with a label       
\draw        [decorate,decoration={brace,amplitude=5pt,raise=1pt}] (1,0) -- (2,0)       node [pos=0.5,anchor=south,yshift=5pt] {$1/g$};     
      
% Add an underbrace below the entire interval with a label       
\draw [decorate,decoration={brace,amplitude=10pt,mirror,raise=2pt}] (0,0) -- (4,0)       node [pos=0.5,anchor=north,yshift=-15pt] {$1/z$};    
     
% Draw a vertical dashed line   
\draw[dashed, blue] (1,0) -- (1,1.25);
\draw[dashed, blue] (3,0) -- (3,1.25) node[midway, anchor=west, xshift=-0.05cm, yshift=-0.2cm] {$\delta/2$};

\end{tikzpicture}     
\caption{Instance with $g>z$. Each bin $B$ produced by $\hat{\pi}_{\mathrm{static}}$ has width $1/g$.}  
\label{fig:N_ge_g} 
\end{figure}\begin{figure}     
\centering     
\begin{tikzpicture}[scale=1.2] % Adjust scale for overall size       

% Draw the horizontal line for the axis       
\draw[thick, -latex] (0,0) -- (12,0) node[anchor=north west] {$x$};            

% Draw ticks for intervals on the axis       
\foreach \x in {2, 4, 6, 8, 10} {         \draw[blue!50, thick] (\x,0.1) -- (\x,-0.1);       }              

% Draw the edges of the triangle red, except for the base
\draw[orange, thick] (0,0) -- (1,1.5);       
\draw[orange, thick] (1,1.5) -- (2,0);            

% Add an upper brace for the triangle interval with a label       
\draw        [decorate,decoration={brace,amplitude=5pt,mirror,raise=1pt}] (0,0) -- (2,0)       node [pos=0.5,anchor=south,yshift=-21pt] {$1/z$};            

% Add an underbrace below the entire interval with a label       
\draw [decorate,decoration={brace,amplitude=10pt,mirror,raise=14pt}] (0,0) -- (6,0)       node [pos=0.5,anchor=north,yshift=-22pt] {$1/g$};
      
% Draw a vertical dashed line        
\draw[dashed, blue] (0.5,0) -- (0.5,0.75);     
\draw[dashed, blue] (1.5,0) -- (1.5,0.75) node[midway, anchor=west, xshift=-0.7cm, yshift=0cm] {$\delta/2$};          
\end{tikzpicture}    
\caption{Instance with $g<z$. Each bin $B$ produced by $\hat{\pi}_{\mathrm{static}}$ has width $1/g$.}     
\label{fig:N_ll_g} 
\end{figure}

\section{Adaptivity to margin parameter $\alpha$}
In this section, we provide some discussions on the possibility of adapting to the margin parameter $\alpha$ if it is unknown. Recall in Section~\ref{sec:algo}, the grid choice of Algorithm~\ref{alg:adaptive-bin} requires knowledge of $\alpha$. One may ask if such knowledge is essential in obtaining small regret. Unfortunately, the following theorem demonstrates that the price of not knowing $\alpha$ is at least a polynomial increase in regret. Denote $\gamma(\alpha^\star)=\beta(\alpha^\star+1)/(2\beta+d)=\gamma^\star$.
\begin{theorem}\label{thm:alpha-lower}

Consider $\beta=d=1$. For any algorithm that does not know the true margin parameter $\alpha^\star$, there exists a choice of $\alpha^\star$  such that
% \[\sup_{\mathcal{F}(\alpha^\star,1)}
% R_{T}(\pi)
% \ge\tilde{D}_3T^{\frac{1-\frac{\alpha^\star+1}{3}}{1-(\frac{\alpha^\star+1}{3})^M}+\kappa_1},\]
\[\sup_{\mathcal{F}(\alpha^\star,1)}
R_{T}(\pi)
\ge\tilde{D}_3
T^{\frac{1-\gamma^\star}{1-(\gamma^{\star})^M} \;+\; \kappa_1}
\cdot(\frac{1}{M})^{\tilde{D}_{4}},\]
for some $\tilde{D}_{3}, \tilde{D}_{4}, \kappa_1>0$ that are independent
of $T$.
\end{theorem}

\noindent See Appendix~\ref{sec:alpha-lower-proof} for the proof. 

\medskip
Theorem~\ref{thm:alpha-lower} is mostly interesting when the number $M$ of batches is small. In particular, when $M$ is a constant, Theorem~\ref{thm:alpha-lower} demonstrates that any algorithm that does not know $\alpha^\star$ will incur a regret that is polynomially larger than the optimal regret attainable by Algorithm~\ref{alg:adaptive-bin} with the knowledge of $\alpha^\star$. 
This result on separation shows that batch learning for nonparametric bandits is indeed much harder than the fully online case to some extent, where adaptivity to $\alpha^\star$ could be achieved for free \cite{perchet2013multi}.

The intuition behind the proof of Theorem~\ref{thm:alpha-lower} is that since the algorithm does not know $\alpha^\star$, it has little hope to pick the first batch size $t_1$ optimally. If $t_1$ is too large, then the adversary can choose a large $\alpha^\star$, which corresponds to the family of reward functions with larger gaps, so that the algorithm's regret during the first batch explodes. If $t_1$ is too small, then the adversary can choose a small $\alpha^\star$, which corresponds to the family of reward functions with smaller gaps.
In this case, the algorithm's knowledge gathered during the first batch is not enough to distinguish the arms and its regret will explode in later batches. 

\section{Discussions}

In this paper, we characterize the fundamental limits of batch learning
in nonparametric contextual bandits. In particular, our optimal batch
learning algorithm (i.e., Algorithm~\ref{alg:adaptive-bin}) is able
to match the optimal regret in the fully online setting with only
$O(\log\log T)$ policy updates. Our work open a few
interesting avenues to explore in the future.

\paragraph{Extensions to multiple arms.} With slight modification,
our algorithm works for nonparametric contextual bandits with more
than two arms. However, it remains unclear what the fundamental limits
of batch learning are in this multi-armed case (i.e., when $K$ is large). 

\paragraph{Improving the log factor.}Comparing the upper and lower
bounds, it is evident that Algorithm~\ref{alg:adaptive-bin} is near-optimal
up to log factors. It is certainly interesting to improve this log
factor, either by strengthening the lower bound, or making the upper
bound more efficient. 

\paragraph{Adapting to the margin parameter.}While we have shown that adaptivity to the margin parameter is not possible when the batch constraint is stringent, i.e., when $M$ is constant, it leaves open the question of designing optimal adaptive algorithm when $M$ is large, say $M \asymp \log \log T$. In fact, when $M=T$, i.e., in the fully online setting, \cite{perchet2013multi} provides an adaptively binned successive elimination algorithm that is capable of adapting to the margin parameter optimally. 

\paragraph{Gap-dependent regret bounds.} The separation between the arms for nonparametric bandits is controlled by the margin parameter $\alpha$. The current paper focuses on the family of non-degenerate nonparametric bandit instances. When $\alpha\rightarrow\infty$, the two arms will have constant gap over the entire covariate
space, meaning the instance is essentially reduced to the multi-armed bandits without covariates. The BaSE policy with geometric grid from \cite{zijun2020batch} can then be applied to obtain a gap-dependent regret bound. Whether one can develop a procedure that bridges these two regimes under the batched setting is an interesting future direction.

\subsection*{Acknowledgements} CM is partially supported by the National Science Foundation via grant DMS-2311127. 

\bibliographystyle{plain}
\bibliography{Batch}

% Clear the page and then start the appendix
\clearpage
\appendix
% Redefine the section command to use letters
\renewcommand{\thesection}{\Alph{section}}

\section{Proof of Theorem~\ref{thm:lower-bound-adaptive}}\label{sec:adaptive-lower-proof}

Define $b\asymp T^{(1-\gamma)/(1-\gamma^{M})}$. For each $1 \leq m \leq M$, we set $T_{m}=\lfloor b^{(1-\gamma^{m})/(1-\gamma)}\rfloor$,  
$z_{m}=\lceil(36T_{m-1}M^{2})^{1/(2\beta+d)}\rceil$. Define the event 
\[A_{m}=\{t_{m-1}<T_{m-1}<T_{m}\le t_{m}\}.\]
Intuitively, $A_{m}$ models the event when the algorithm's
selected grid points $t_{m-1}$ and $t_{m}$ are suboptimal. 
When $A_m$ happens, the goal is to design a problem instance such that using observations up $t_{m-1}$
cannot distinguish the optimal arm and therefore the policy
must incur a large regret between $t_{m-1}$ and $t_{m}$.

\subsection{Construction of the hard instances}

Consider a regular partition of $[0,1]^{d}$ into $M^{d}$ bins of
equal width. Denote the bins by $C_{i}$ for $i=1,\dots,M^{d}$. For
any $m\in[M],$ define $I_{m}=\{i\in[M^{d}]:(m-1)M^{d-1}+1\le i\le mM^{d-1}\}$.
Given $i\in I_{m}$, we further consider a regular partition of $C_{i}$
into $z_{m}^{d}$ bins, which are denoted by $\binm$ for $j=1,\dots,z_{m}^{d}$.
Denote by $\centerm$ the center of $\binm$.

Define $s_{m}=\lceil(M^{-\alpha_{}\beta})z_{m}^{d-\alpha_{}\beta}\rceil$
and $s=M^{d-1}\sum_{m=1}^{M}s_{m}$. Denote by $\Omega_{m}=\{\pm1\}^{s_{m}}$.
Define a set of binary sequences
\[
\Omega=\{\omega\in\{\pm1\}^{s}:1\le m\le M,i\in I_{m},\omega_{i}\in\Omega_{m}\},
\]
 where for $i\in I_m$, $\omega_i$ stands for the $i$-th block in $\omega\in\{\pm1\}^{s}$ of size $s_m$. For any $\omega\in\Omega$, define a function $f_{\omega}:[0,1]^{d}\mapsto\mathbb{R}$:
\[
f_{\omega}(x)=\frac{1}{2}+\sum_{m=1}^{M}\sum_{i\in I_{m}}\sum_{j=1}^{s_{m}}\omega_{i,j}^{m}\xi_{i,j}^{m}(x),
\]
where $\xi_{i,j}^{m}(x)=D_{\phi}(Mz_{m})^{-\beta}\phi(2Mz_m(x-\centerm))\mathbf{1}\{x\in \binm\}$. It is straightforward to check $f_\omega$ satisfies the smoothness condition.

Define the family of reward instances
\[
\mathcal{C}\coloneqq\left\{ f_{1}(x)=f_{\omega}(x),f_{-1}(x)=\frac{1}{2}\mid\omega\in\Omega\right\} .
\]
We now verify the margin condition. One has
\begin{align*}
    P_{X}\left(0<\left|f_{\omega}(X)-\frac{1}{2}\right|\leq D_\phi\delta\right)
    &= \sum_{m=1}^{M}\sum_{i\in I_{m}}\sum_{j=1}^{s_{m}}P_{X}(0<\left|f_{\omega}(X)-\frac{1}{2}\right|\leq D_\phi\delta,X\in\binm)\\
    &= \sum_{m=1}^{M}\sum_{i\in I_{m}}\sum_{j=1}^{s_{m}}P_{X}(0<\phi(2Mz_m(x-\centerm))\leq \delta(Mz_{m})^{\beta},X\in\binm)\\
    &\le\sum_{m=1}^{M}\sum_{i\in I_{m}}M^{-d}(1+d)\delta^\alpha=(1+d)\delta^\alpha,
\end{align*}
where the penultimate step applies Lemma~\ref{lemma:margin-check}. Hence, the margin assumption is satisfied, and $\mathcal{C}\subseteq\mathcal{F}(\alpha,\beta)$.

\subsection{Lower bounding the regret during the $m$-th batch}

Since the worst-case inferior sampling rate is lower bounded by the average over the family $\Omega$, 
\begin{align}
\sup_{(f,\frac{1}{2})\in\instfam(\alpha,\beta)}&S_{T_m}(\pi,f)\nonumber \\
&\ge\mathbb{E}_{\omega\sim\omgdist}\eoverpi\left[\sum_{t=1}^{T_m}\mathbf{1}\{\pi_t(X_t)\neq\pi^\star(X_t),f_\omega(X_t)\neq\frac{1}{2}\}\right]\nonumber \\
 & \overset{\mathrm{(i)}}{\ge}\sum_{t=T_{m-1}+1}^{T_{m}}\sum_{i\in I_{m}}\sum_{j=1}^{s_{m}}\mathbb{E}_{\omega\sim\omgdist}\eoverpi^{t}\left[\mathbf{1}\{X_{t}\in\binm,\pi_{t}(X_{t})\neq\omega_{i,j}^{m}\}\right]\nonumber \\
 % & =D_{\phi}(Mz_{m})^{-\beta}\sum_{t=T_{m-1}+1}^{T_{m}}\sum_{i\in I_{m}}\sum_{j=1}^{s_{m}}\frac{1}{2^{s}}\sum_{\omega\in\Omega}\eoverpi^{t}\left[\mathbf{1}\{X_{t}\in\peakm,\pi_{t}(X_{t})\neq\omega_{i,j}^{m}\}\right]\nonumber \\
 & =(Mz_{m})^{-d}\sum_{t=T_{m-1}+1}^{T_{m}}\sum_{i\in I_{m}}\sum_{j=1}^{s_{m}}\frac{1}{2^{s}}\sum_{\wloo\in\omgloo}\underbrace{\sum_{l\in\{\pm1\}}\mathbb{E}_{\pi,\omega_{i,j}^{m}=l}^{t}P_{X}(\pi_{t}(X_{t})\neq l\mid X_{t}\in\binm)}_{U_{m,i,j}^{t}}.\label{eq:regret-to-test}
\end{align}
Here, step (i) uses the fact that regret is only incurred on $\binm$'s and the optimal action is specified by $\omega_{i,j}^m$. Besides, we use $\omega_{-(i,j)}$ to represent the vector after leaving out the $j$-th entry in the $i$-th block of $\omega$.
By Le Cam's method, one has
\begin{align*}
U_{m,i,j}^{t} & \ge1-\|\pminus^{t}-\pplus^{t}\|_{\mathrm{TV}}\\
 & \ge1-\|\pminus^{T_{m}}-\pplus^{T_{m}}\|_{\mathrm{TV}}\\
 & =\int\min\left\{ \mathrm{d}\pminus^{T_{m}},\mathrm{d}\pplus^{T_{m}}\right\} \\
 & \ge\int_{A_{m}}\min\left\{ \mathrm{d}\pminus^{T_{m}},\mathrm{d}\pplus^{T_{m}}\right\} ,
\end{align*}
where the second inequality  is due to $t\le T_m$. Since the available
observations for $\pi$ at $T_{m}$ are the same as those at $T_{m-1}$ under $A_{i}$, we continue to lower bound
\begin{align*}
U_{m,i,j}^{t}  
 &\ge\int_{A_{m}}\min\left\{ \mathrm{d}\pminus^{T_{m-1}},\mathrm{d}\pplus^{T_{m-1}}\right\} \\
 &=\frac{1}{2}\int_{A_{m}}\left(\mathrm{d}\pminus^{T_{m-1}}+\mathrm{d}\pplus^{T_{m-1}}-|\mathrm{d}\pminus^{T_{m-1}}-\mathrm{d}\pplus^{T_{m-1}}|\right)\\
 &\ge\frac{1}{2}\left(\pminus^{T_{m-1}}(A_{m})+\pplus^{T_{m-1}}(A_{m})\right)-\|\pminus^{T_{m-1}}-\pplus^{T_{m-1}}\|_{\mathrm{TV}}\\
 &\ge\frac{1}{2}\left(\pminus(A_{m})+\pplus(A_{m})\right)-\frac{1}{2M},
\end{align*}
where the last step applies Lemma~\ref{lemma:single-bin-tv}. Plugging the above back to~(\ref{eq:regret-to-test}), we obtain

\begin{align*}
\sup_{f\in\instfam(\alpha,\beta)}& S_{T_m}(\pi,f) \nonumber\\
& \ge (Mz_{m})^{-d}\sum_{t=T_{m-1}+1}^{T_{m}}\sum_{i\in I_{m}}\sum_{j=1}^{s_{m}}\frac{1}{2^{s+1}}\sum_{\wloo\in\omgloo}\left(\pminus(A_{m})+\pplus(A_{m})-\frac{1}{M}\right)\nonumber \\
 & =(Mz_{m})^{-d}\sum_{t=T_{m-1}+1}^{T_{m}}\sum_{i\in I_{m}}\sum_{j=1}^{s_{m}}\frac{1}{2}\left(\mathbb{E}_{\omega\sim\omgdist}\mathbb{P}_{\pi,\omega}(A_{m})-\frac{1}{2M}\right)\nonumber\\
 & =(Mz_{m})^{-d}(T_{m}-T_{m-1})M^{d-1}s_{m}\left(\mathbb{E}_{\omega\sim\omgdist}\mathbb{P}_{\pi,\omega}(A_{m})-\frac{1}{2M}\right)\nonumber\\
 % & =\frac{1}{2}D_{\phi}(Mz_{m})^{-(\beta+d)}(T_{m}-T_{m-1})M^{d-1}(M^{-\alpha_{m}\beta})z_{m}^{d-\alpha_{m}\beta}\left(\mathbb{E}_{\omega\sim\omgdist}\mathbb{P}_{\pi,\omega}(A_{m})-\frac{c}{M}\right)\\
 & \asymp M^{-1-\alpha\beta}T_{m}z_{m}^{-\alpha\beta}\left(\mathbb{E}_{\omega\sim\omgdist}\mathbb{P}_{\pi,\omega}(A_{m})-\frac{1}{2M}\right)\nonumber.
\end{align*}
By Lemma~\ref{inferior-to-regret}, we obtain
\begin{align}\label{eq:reg-to-m}
       \sup_{f\in\instfam(\alpha,\beta)} R_{T_m}(\pi,f)
     &\apprge T_{m}^{-\frac{1}{\alpha}}\left[\sup_{(f,\frac{1}{2})\in\mathcal{F}(\alpha,\beta)}S_{T_{m}}(\pi;f)\right]^{\frac{1+\alpha}{\alpha}}\nonumber\\
     &\apprge M^{-\frac{(\alpha\beta+1)(1+\alpha)}{\alpha}}T_{m}z_{m}^{-\beta(1+\alpha)}\left(\mathbb{E}_{\omega\sim\omgdist}\mathbb{P}_{\pi,\omega}(A_{m})-\frac{1}{2M}\right)^{\frac{1+\alpha}{\alpha}}.
\end{align}
Since $\sum_{k=1}^M\mathbb{E}_{\omega\sim\omgdist}\mathbb{P}_{\pi,\omega}(A_{k})\ge1$, there exists some $m^\star\in[M]$ such that $\mathbb{E}_{\omega\sim\omgdist}\mathbb{P}_{\pi,\omega}(A_{m^\star})\ge1/M$. Because relation~\eqref{eq:reg-to-m}
 holds for any $m\in[M]$, we reach
 \begin{align*}
     \sup_{f\in\instfam(\alpha,\beta)} R_{T}(\pi,f)
     &\apprge M^{-\frac{(\alpha\beta+1)(1+\alpha)}{\alpha}}T_{m^\star}z_{m^\star}^{-\beta(1+\alpha)}\left(\mathbb{E}_{\omega\sim\omgdist}\mathbb{P}_{\pi,\omega}(A_{m^\star})-\frac{1}{2M}\right)^{\frac{1+\alpha}{\alpha}}\\
     &\ge M^{-\frac{(\alpha\beta+1)(1+\alpha)}{\alpha}}T_{m^\star}z_{m^\star}^{-\beta(1+\alpha)}\left(\frac{1}{M}-\frac{1}{2M}\right)^{\frac{1+\alpha}{\alpha}}\\
     &\asymp M^{-\frac{(\alpha\beta+2)(1+\alpha)}{\alpha}}T_{m^\star}z_{m^\star}^{-\beta(1+\alpha)}
     \apprge M^{-\frac{(\alpha\beta+2)(1+\alpha)}{\alpha}}\frac{T_{m^\star}}{T^\gamma_{m^\star-1}}M^{-2\gamma}\apprge (\frac{1}{M})^{\tilde{D}_2}\cdot T^{\frac{1-\gamma}{1-\gamma^M}}.
 \end{align*}

\subsection{Proof of helper lemmas}
\begin{lemma}\label{lemma:kl-neg-pos}Fix $z>0$ and suppose $f_{\omega}\in\mathcal{C}_{z}$.
For any $n\in[T]$ and any policy $\pi$, one has
\[
\mathrm{KL}(\mathbb{P}_{\pi,f_{\omega_{[-j]}^{-1}}}^{n},\mathbb{P}_{\pi,f_{\omega_{[-j]}^{1}}}^{n})\le2nz^{-(2\beta+d)}.
\]

\end{lemma}

\begin{proof}We can compute
\begin{align*}
\mathrm{KL}(\mathbb{P}_{\pi,f_{\omega_{[-j]}^{-1}}}^{n},\mathbb{P}_{\pi,f_{\omega_{[-j]}^{1}}}^{n}) & \overset{\mathrm{(i)}}{\le}8\mathbb{E}_{\pi,f_{\omega_{[-j]}^{-1}}}[\sum_{t=1}^{n}(f_{\omega_{[-j]}^{-1}}(X_{t})-f_{\omega_{[-j]}^{1}}(X_{t}))^{2}\mathbf{1}\{\pi_{t}(X_{t})=1\}]\\
 & \overset{\mathrm{(ii)}}{\le}32D_{\phi}^{2}z^{-2\beta}\mathbb{E}_{\pi,f_{\omega_{[-j]}^{-1}}}[\sum_{t=1}^{n}\mathbf{1}\{\pi_{t}(X_{t})=1,X_{t}\in C_{j}\}]\\
 & \overset{\mathrm{(iii)}}{=}32D_{\phi}^{2}z^{-(2\beta+d)}\sum_{t=1}^{n}\mathbb{P}_{\pi,f_{\omega_{[-j]}^{-1}}}^{t}(\pi_{t}(X_{t})=1\mid X_{t}\in C_{j})\\
 & \overset{\mathrm{(iv)}}{\le}32D_{\phi}^{2}z^{-(2\beta+d)}n\le2nz^{-(2\beta+d)}.
\end{align*}
Here, step (i) uses the standard decomposition of KL divergence and
Bernoulli reward structure; step (ii) is due to the definition of
$f_{\omega}$; step (iii) uses $\mathbb{P}(X_{t}\in C_{j})=1/z^{d}$,
and step (iv) arises from $\mathbb{P}_{\pi,f_{\omega_{[-j]}^{-1}}}^{t}(\pi_{t}(X_{t})=1\mid X_{t}\in C_{j})\le1$
for any $1\le t\le n$. 
\end{proof}

\begin{lemma}\label{lemma:margin-check}
    Fix $m\in[M]$. For any $i\in I_m$, one has
    \[\sum_{j=1}^{s_{m}}P_{X}(0<\phi(2Mz_m(x-\centerm))\leq \delta(Mz_{m})^{\beta},X\in\binm)
    \le M^{-d}(1+d)\delta^\alpha.\]
\end{lemma}
\begin{proof}
    By direct calculation,
    \begin{align*}
        \sum_{j=1}^{s_{m}}
        &P_{X}(0<\phi(2Mz_m(x-\centerm))\leq \delta(Mz_{m})^{\beta},X\in\binm)\\
        &=s_m(Mz_m)^{-d}\int_{[0,1]^d}\mathbf{1}\{\phi(x)\le\delta(Mz_m)^\beta\}\mathsf{d}x\\
        &\le s_m(Mz_m)^{-d}\mathbf{1}\{\delta(Mz_m)^\beta>1\}+s_m(Mz_m)^{-d}d\delta^{1/\beta}(Mz_m)\mathbf{1}\{\delta(Mz_m)^\beta\le1\}\\
        &\le M^{-d}\delta^\alpha+dM^{-d}\delta^\alpha=M^{-d}(1+d)\delta^\alpha,
    \end{align*}
    where the first inequality is due to $\int_{[0,1]^d}\mathbf{1}\{\phi(x)\le\delta(Mz_m)^\beta\}\mathsf{d}x=1-(1-\delta^{1/\beta}(Mz_m))^d\le d\delta^{1/\beta}(Mz_m)$ when $\delta(Mz_m)^\beta\le1$; the second inequality uses the definition of the indicators and the assumption $\alpha\beta\le1$.
\end{proof}

\begin{lemma}\label{lemma:single-bin-tv}
   Fix any $n\in[T]$ and any policy $\pi$. For any $m\in[M],i\in I_m, j\in [s_m]$, 
   \[\|\pminus^{n}-\pplus^{n}\|_\mathrm{TV}\le \sqrt{n(Mz_m)^{-(2\beta+d)}}.\]
\end{lemma}
\begin{proof}It suffices to bound their KL-divergence. We can compute
\begin{align*}
\mathrm{KL}(\pminus^{n},\pplus^{n}) & \overset{\mathrm{(i)}}{\le}8\mathbb{E}_{\pi,\omega_{i,j}^m=-1}[\sum_{t=1}^{n}(f_{\omega_{i,j}^m=-1}(X_{t})-f_{\omega_{i,j}^m=1}(X_{t}))^{2}\mathbf{1}\{\pi_{t}(X_{t})=1\}]\\
 & \overset{\mathrm{(ii)}}{\le}32D_{\phi}^{2}(Mz_m)^{-2\beta}\mathbb{E}_{\pi,\omega_{i,j}^m=-1}[\sum_{t=1}^{n}\mathbf{1}\{\pi_{t}(X_{t})=1,X_{t}\in \binm\}]\\
 & \overset{\mathrm{(iii)}}{=}32D_{\phi}^{2}(Mz_m)^{-(2\beta+d)}\sum_{t=1}^{n}\mathbb{P}_{\pi,\omega_{i,j}^m=-1}^{t}(\pi_{t}(X_{t})=1\mid X_{t}\in \binm)\\
 & \overset{\mathrm{(iv)}}{\le}32D_{\phi}^{2}(Mz_m)^{-(2\beta+d)}n\le2n(Mz_m)^{-(2\beta+d)}.
\end{align*}
Here, step (i) uses the standard decomposition of KL divergence and
Bernoulli reward structure; step (ii) is due to the definition of
$f_{\omega}$; step (iii) uses $\mathbb{P}(X_{t}\in \binm)=1/(Mz_m)^{d}$,
and step (iv) arises from $\mathbb{P}_{\pi,\omega_{i,j}^m=-1}^{t}(\pi_{t}(X_{t})=1\mid X_{t}\in \binm)\le1$
for any $1\le t\le n$. 

By Pinsker's inequality, 
\[\|\pminus^{n}-\pplus^{n}\|_\mathrm{TV}\le\sqrt{\frac{1}{2}\mathrm{KL}(\pminus^{n},\pplus^{n})}\le\sqrt{n(Mz_m)^{-(2\beta+d)}}.\]
\end{proof}

\section{Proof of Theorem~\ref{thm:upper-bound}}\label{sec:upper-proof}

Our proof of Theorem~\ref{thm:upper-bound} is inspired by the framework developed
in \cite{perchet2013multi}. Our setting presents additional technical
difficulty due to the batch constraint. 

We begin with introducing some useful notations. Recall the tree growing
process described in section~\ref{sec:algo}, where we have defined
a tree $\mathcal{T}$ of depth $M$. The root (depth 0) of the tree
is the whole space $\mathcal{X}$. In depth $1$, $\mathcal{X}$ has
$g_{0}^{d}$ children, each of which is a bin of width $1/g_{0}$.
For each bin in depth 1, it has $g_{1}^{d}$ children, each of which
is a bin of width $1/(g_{0}g_{1})$. These children form the depth
2 nodes of the tree $\mathcal{T}$. We form the tree recursively until
depth $M$. 

For a bin $C\in\mathcal{T}$, we define its parent by $\mathsf{p}(C)=\{C^{\prime}\in\mathcal{T}:C\in\mathsf{child}(C^{\prime})\}$.
Moreover, we let $\mathsf{p}^{1}(C)=\mathsf{p}(C)$ and define $\mathsf{p}^{k}(C)=\mathsf{p}(\mathsf{p}^{k-1}(C))$
for $k\geq2$ recursively. In all, we denote by $\mathcal{P}(C)=\{C^{\prime}\in\mathcal{T}:C^{\prime}=\mathsf{p}^{k}(C)\textrm{ for some }k\ge1\}$
all the ancestors of the bin $C$. 

We also define $\mathcal{L}_{t}$ to be the set of active bins at
time $t$, with the dummy case $\mathcal{L}_{0}=\{\mathcal{X}\}$.
Clearly, for $1\leq t\leq t_{1}$, one has $\mathcal{L}_{1}=\mathcal{B}_{1}$,
where $\mathcal{B}_{1}$ are all the bins in the first layer. 

\subsection{Two clean events}

The regret analysis relies on two clean events. First, fix a batch
$i\ge1$, and recall $\mathcal{L}_{t_{i-1}+1}$ is the set of active
bins at time $t_{i-1}+1$. We denote the random number of pulls for
a bin $C\in\mathcal{L}_{t_{i-1}+1}$ within batch $i$ to be 
\[
m_{C,i}\coloneqq\sum_{t=t_{i-1}+1}^{t_{i}}\mathbf{1}\{X_{t}\in C\}.
\]
Clearly, it has expectation 
\[
m_{C,i}^{\star}=\mathbb{E}[m_{C,i}]=(t_{i}-t_{i-1})\mathbb{P}_{X}(X\in C).
\]
The first clean event claims that $m_{C,i}$ concentrates well around
its expectation $m_{C,i}^{\star}$ uniformly over all $C\in\mathcal{T}$.
We denote this event by $E$. 

\begin{lemma} \label{lemma:clean-event-1} 

Suppose that $M\le D_{1}\log(T)$ for some constant $D_{1}>0$. With
probability at least $1-1/T$, for all $1\leq i\leq M$, and $C\in\mathcal{L}_{t_{i-1}+1}$,
we have 
\[
\frac{1}{2}m_{C,i}^{\star}\le m_{C,i}\le\frac{3}{2}m_{C,i}^{\star}.
\]
\end{lemma}

\noindent See Section~\ref{subsec:proof-lemma-clean1} for the proof. 

\medskip

Since $M\le D_{1}\log(T)$ by assumption, we can apply Lemma~\ref{lemma:clean-event-1}
to obtain 
\[
\mathbb{E}[R_{T}(\hat{\pi})\mathbf{1}(E^{c})]\le T\mathbb{P}(E^{c})=1.
\]
Therefore, in the remaining proof, we condition on $E$ and focus
on bounding $\mathbb{E}[R_{T}(\hat{\pi})\mathbf{1}(E)]$. 

The second clean event is on the elimination process. Since we use
successive elimination in each bin, it is natural to expect that the
optimal arm in each bin is not eliminated during the process. To mathematically
specify this event, we need a few notations. 

For each bin $C\in\mathcal{L}_{i}$, let $\mathcal{I}_{C}^{\prime}$
be the set of remaining arms at the end of batch $i$, i.e., after
Algorithm~\ref{algo-subroutine} is invoked. Define
\begin{align*}
\bar{\mathcal{I}}_{C} & =\left\{ k\in\{1,-1\}:\sup_{x\in C}f^{\star}(x)-f^{(k)}(x)\le c_{1}|C|^{\beta}\right\} ,
\end{align*}
\[
\underline{\mathcal{I}}_{C}=\left\{ k\in\{1,-1\}:\sup_{x\in C}f^{\star}(x)-f^{(k)}(x)\le c_{0}|C|^{\beta}\right\} ,
\]
where $c_{0}=2Ld^{\beta/2}+1$ and $c_{1}=8c_{0}$. Clearly, we have
\[
\underline{\mathcal{I}}_{C}\subseteq\bar{\mathcal{I}}_{C}.
\]
Define a good event $\mathcal{A}_{C}=\{\underline{\mathcal{I}}_{C}\subseteq\mathcal{I}_{C}^{\prime}\subseteq\bar{\mathcal{I}}_{C}\}$,
which is the event that the remaining arms in $C$ have gaps of correct
order. In addition, define $\mathcal{G}_{C}=\cap_{C^{\prime}\in\mathcal{P}(C)}\mathcal{A}_{C^{\prime}}$.
Recall $\mathcal{B}_{i}$ is the set of bins $C$ with $|C|=(\prod_{l=0}^{i-1}g_{l})^{-1}=w_{i}$
for $i\ge1$.

\begin{lemma}\label{lemma:clean-event-2}For any $1\le i\le M-1$
and $C\in\mathcal{B}_{i}$, we have 
\[
\mathbb{P}(E\cap\mathcal{G}_{C}\cap\mathcal{A}_{C}^{c})\leq\frac{4m_{C,i}^{\star}}{T|C|^{d}}.
\]
\end{lemma}

\noindent In words, Lemma~\ref{lemma:clean-event-2} guarantees
that $\mathcal{A}_{C}$ happens with high probability if $E$ holds
and $\mathcal{A}_{C'}$ holds for all the ancestors $C\text{\textquoteright}$
of $C$. See Section~\ref{subsec:proof-lemma-clean2} for the proof. 

\subsection{Regret decomposition}

In this section, we decompose the regret into three terms. First,
for a bin $C$, we define 
\[
r_{T}^{\textrm{live}}(C)\coloneqq\sum_{t=1}^{T}\left(f^{\star}(X_{t})-f^{(\pi_{t}(X_{t}))}(X_{t})\right)\mathbf{1}(X_{t}\in C)\mathbf{1}(C\in\mathcal{L}_{t}).
\]
In addition, define $\mathcal{J}_{t}\coloneqq\cup_{s\le t}\mathcal{L}_{s}$
to be the set of bins that have been live up until time $t$. Correspondingly
we define 
\[
r_{T}^{\textrm{born}}(C)\coloneqq\sum_{t=1}^{T}\left(f^{\star}(X_{t})-f^{(\pi_{t}(X_{t}))}(X_{t})\right)\mathbf{1}(X_{t}\in C)\mathbf{1}(C\in\mathcal{J}_{t}).
\]
It is clear from the definition that for any $C\in\mathcal{T}$, one
has 
\begin{align*}
r_{T}^{\textrm{born}}(C) & =r_{T}^{\textrm{live}}(C)+\sum_{C^{\prime}\in\mathsf{child}(C)}r_{T}^{\textrm{born}}(C^{\prime})\\
 & =r_{T}^{\textrm{born}}(C)\mathbf{1}(\mathcal{A}_{C}^{c})+r_{T}^{\textrm{live}}(C)\mathbf{1}(\mathcal{A}_{C})+\sum_{C^{\prime}\in\mathsf{child}(C)}r_{T}^{\textrm{born}}(C^{\prime})\mathbf{1}(\mathcal{A}_{C}).
\end{align*}
Applying this relation recursively leads to the following regret decomposition:
\begin{align*}
R_{T}(\pi) & =\mathbb{E}[r_{T}^{\textrm{born}}(\mathcal{X})]\\
 & =\underbrace{\mathbb{E}[r_{T}^{\textrm{live}}(\mathcal{X})]}_{=0}+\sum_{C^{\prime}\in\mathsf{child}(\mathcal{X})}\mathbb{E}[r_{T}^{\textrm{born}}(C^{\prime})]\\
 & =\sum_{1\le i<M}\left(\underbrace{\sum_{C\in\mathcal{B}_{i}}\mathbb{E}[r_{T}^{\textrm{born}}(C)\mathbf{1}(\mathcal{G}_{C}\cap\mathcal{A}_{C}^{c})]}_{\eqqcolon U_{i}}+\underbrace{\sum_{C\in\mathcal{B}_{i}}
 \mathbb{E}[r_{T}^{\textrm{live}}(C)\mathbf{1}(\mathcal{G}_{C}\cap\mathcal{A}_{C})]}_{\eqqcolon V_{i}}\right)\\
 & \quad+\sum_{C\in\mathcal{B}_{M}}
 \mathbb{E}[r_{T}^{\textrm{live}}(C)\mathbf{1}(\mathcal{G}_{C})],
\end{align*}
where the second equality arises from the fact that $r_{T}^{\textrm{live}}(\mathcal{X})=0$.
Indeed, $\mathcal{X}\notin\mathcal{L}_{t}$ for any $1\leq t\leq T$. 

\subsection{Controlling three terms}

In what follows, we control $V_{i},U_{i}$ and the last batch separately. 

\subsubsection{Controlling $V_{i}$}

Fix some $1\leq i\leq M-1$, and some bin $C\in\mathcal{B}_{i}$.
On the event $\mathcal{G}_{C}$ we have $\mathcal{I}_{\mathsf{p}(C)}^{'}\subseteq\bar{\mathcal{I}}_{\mathsf{p}(C)}$,
that is, for any $k\in\mathcal{I}_{\mathsf{p}(C)}^{'}$, 
\[
\sup_{x\in\mathsf{p}(C)}f^{\star}(x)-f^{(k)}(x)\le c_{1}|\mathsf{p}(C)|^{\beta}.
\]
This implies that for any $x\in C$, and $k\in\mathcal{I}_{\mathsf{p}(C)}^{'}$,
\begin{equation}
\left(f^{\star}(x)-f^{(k)}(x)\right)\bm{1}\{\mathcal{G}_{C}\}\leq c_{1}|\mathsf{p}(C)|^{\beta}\mathbf{1}(0<\left|f^{(1)}(x)-f^{(-1)}(x)\right|\le c_{1}|\mathsf{p}(C)|^{\beta}).\label{eq:remaining-arm-gap}
\end{equation}
As a result, we obtain
\begin{align*}
& \mathbb{E}[r_{T}^{\textrm{live}}(C)\mathbf{1}(\mathcal{G}_{C}\cap\mathcal{A}_{C})]  =\mathbb{E}\left[\sum_{t=1}^{T}\left(f^{\star}(X_{t})-f^{(\pi_{t}(X_{t}))}(X_{t})\right)\mathbf{1}(X_{t}\in C)\mathbf{1}(C\in\mathcal{L}_{t})\mathbf{1}(\mathcal{G}_{C}\cap\mathcal{A}_{C})\right]\\
 &\quad  \overset{\mathrm{(i)}}{\le}\mathbb{E}\left[\sum_{t=1}^{T}c_{1}|\mathsf{p}(C)|^{\beta}\mathbf{1}(0<\left|f^{(1)}(X_{t})-f^{(-1)}(X_{t})\right|\le c_{1}|\mathsf{p}(C)|^{\beta})\mathbf{1}(X_{t}\in C,C\in\mathcal{L}_{t})\mathbf{1}(\mathcal{G}_{C}\cap\mathcal{A}_{C})\right]\\
 &\quad \overset{\mathrm{(ii)}}{\le}c_{1}|\mathsf{p}(C)|^{\beta}\mathbb{E}\left[\sum_{t=t_{i-1}+1}^{t_{i}}\mathbf{1}(0<\left|f^{(1)}(X_{t})-f^{(-1)}(X_{t})\right|\le c_{1}|\mathsf{p}(C)|^{\beta},X_{t}\in C)\mathbf{1}(\mathcal{G}_{C}\cap\mathcal{A}_{C})\right]\\
 &\quad \overset{\mathrm{(iii)}}{\le}c_{1}|\mathsf{p}(C)|^{\beta}\sum_{t=t_{i-1}+1}^{t_{i}}\mathbb{P}(0<\left|f^{(1)}(X_{t})-f^{(-1)}(X_{t})\right|\le c_{1}|\mathsf{p}(C)|^{\beta},X_{t}\in C)\\
 &\quad =c_{1}|\mathsf{p}(C)|^{\beta}(t_{i}-t_{i-1})\mathbb{P}(0<\left|f^{(1)}(X)-f^{(-1)}(X)\right|\le c_{1}|\mathsf{p}(C)|^{\beta},X\in C).
\end{align*}
Here, step (i) uses relation (\ref{eq:remaining-arm-gap}), and the
fact that $\pi_{t}(X_{t})\in\mathcal{I}_{\mathsf{p}(C)}^{'}$ when
$X_{t}\in C$. For step (ii), if $C$ is split, then it is no longer
live, so the live regret incurred on the remaining batches is zero.
On the other hand, if $C$ is not split, then $|\mathcal{I}_{C}^{\prime}|=1$.
Without loss of generality, assume that arm $-1$ is eliminated. Conditioned
on $\mathcal{A}_{C}$, this means $-1\notin\underline{\mathcal{I}}_{C}$
and there exists $x_{0}\in C$ such that $f^{(1)}(x_{0})-f^{(-1)}(x_{0})>c_{0}|C|^{\beta}$.
By the smoothness condition, having a gap at least $c_{0}|C|^{\beta}$
on a single point in $C$ implies $f^{(1)}(x)-f^{(-1)}(x)>|C|^{\beta}$
for all $x\in C$. Therefore, arm 1 which is the remaining one is
the optimal arm for all $x\in C$ and would not incur any regret further.
The third inequality holds since $\mathbf{1}(\mathcal{G}_{C}\cap\mathcal{A}_{C})\le1$.

Taking the sum over all bins in $\mathcal{B}_{i}$ and using the fact
that $|\mathsf{p}(C)|=w_{i-1}$, we obtain 

\begin{align}
\sum_{C\in\mathcal{B}_{i}}\mathbb{E}[r_{T}^{\textrm{live}}(C)\mathbf{1}(\mathcal{G}_{C}\cap\mathcal{A}_{C})] & \le\sum_{C\in\mathcal{B}_{i}}c_{1}w_{i-1}^{\beta}(t_{i}-t_{i-1})\mathbb{P}(0<\left|f^{(1)}(X)-f^{(-1)}(X)\right|\le c_{1}|\mathsf{p}(C)|^{\beta},X\in C)\nonumber \\
 & =c_{1}w_{i-1}^{\beta}(t_{i}-t_{i-1})\sum_{C\in\mathcal{B}_{i}}\mathbb{P}(0<\left|f^{(1)}(X)-f^{(-1)}(X)\right|\le c_{1}w_{i-1}^{\beta},X\in C).\label{eq:vi-calculation}
\end{align}
Note that
\begin{align}
\sum_{C\in\mathcal{B}_{i}}\mathbb{P}(0<\left|f^{(1)}(X)-f^{(-1)}(X)\right|\le c_{1}w_{i-1}^{\beta},X\in C) & =\mathbb{P}(0<\left|f^{(1)}(X)-f^{(-1)}(X)\right|\le c_{1}w_{i-1}^{\beta})\nonumber \\
 & \le D_{0}\cdot\left[c_{1}w_{i-1}^{\beta}\right]^{\alpha},\label{eq:margin-condition}
\end{align}
where the last inequality follows from the margin condition. Combining
relations (\ref{eq:margin-condition}) and (\ref{eq:vi-calculation}),
we reach
\begin{align*}
\sum_{C\in\mathcal{B}_{i}}\mathbb{E}[r_{T}^{\textrm{live}}(C)\mathbf{1}(\mathcal{G}_{C}\cap\mathcal{A}_{C})] & \le(t_{i}-t_{i-1})\cdot[c_{1}w_{i-1}^{\beta}]^{1+\alpha}\cdot D_{0}.
\end{align*}

\subsubsection{Controlling $U_{i}$}

Fix some $1\leq i\leq M-1$, and some bin $C\in\mathcal{B}_{i}$.
Again, using the definition of $\mathcal{G}_{C}$, we obtain 

\begin{align*}
\mathbb{E}[r_{T}^{\textrm{born}}(C)\mathbf{1}(\mathcal{G}_{C}\cap\mathcal{A}_{C}^{c})] & =\mathbb{E}\left[\sum_{t=1}^{T}\left(f^{\star}(X_{t})-f^{(\pi_{t}(X_{t}))}(X_{t})\right)\mathbf{1}(X_{t}\in C)\mathbf{1}(C\in\mathcal{J}_{t})\mathbf{1}(\mathcal{G}_{C}\cap\mathcal{A}_{C}^{c})\right]\\
 & \le\mathbb{E}\left[\sum_{t=1}^{T}c_{1}|\mathsf{p}(C)|^{\beta}\mathbf{1}(0<\left|f^{(1)}(X_{t})-f^{(-1)}(X_{t})\right|\le c_{1}|\mathsf{p}(C)|^{\beta})\mathbf{1}(X_{t}\in C,C\in\mathcal{J}_{t})\mathbf{1}(\mathcal{G}_{C}\cap\mathcal{A}_{C}^{c})\right]\\
 & \le c_{1}|\mathsf{p}(C)|^{\beta}T\mathbb{P}(0<\left|f^{(1)}(X)-f^{(-1)}(X)\right|\le c_{1}|\mathsf{p}(C)|^{\beta},X\in C)\mathbb{P}(\mathcal{G}_{C}\cap\mathcal{A}_{C}^{c}).
\end{align*}
Apply Lemma~\ref{lemma:clean-event-2} to see that 
\begin{align*}
\mathbb{E}[r_{T}^{\textrm{born}}(C)\mathbf{1}(\mathcal{G}_{C}\cap\mathcal{A}_{C}^{c})] & \le c_{1}|\mathsf{p}(C)|^{\beta}T\mathbb{P}(0<\left|f^{(1)}(X)-f^{(-1)}(X)\right|\le c_{1}|\mathsf{p}(C)|^{\beta},X\in C)\frac{4m_{C,i}^{\star}}{T|C|^{d}}\\
 & =c_{1}w_{i-1}^{\beta}\mathbb{P}(0<\left|f^{(1)}(X)-f^{(-1)}(X)\right|\le c_{1}w_{i-1}^{\beta},X\in C)\frac{4(t_{i}-t_{i-1})\mathbb{P}_{X}(X\in C)}{|C|^{d}}\\
 & \le4\bar{c}c_{1}w_{i-1}^{\beta}\mathbb{P}(0<\left|f^{(1)}(X)-f^{(-1)}(X)\right|\le c_{1}w_{i-1}^{\beta},X\in C)(t_{i}-t_{i-1}),
\end{align*}
where we use the fact that $\mathbb{P}_{X}(X\in C)\le\bar{c}|C|^{d}$
in the second inequality. Summing over all bins in $\mathcal{B}_{i}$,
we obtain
\begin{align*}
\sum_{C\in\mathcal{B}_{i}}\mathbb{E}[r_{T}^{\textrm{born}}(C)\mathbf{1}(\mathcal{G}_{C}\cap\mathcal{A}_{C}^{c})] & \le4\bar{c}c_{1}w_{i-1}^{\beta}(t_{i}-t_{i-1})\sum_{C\in\mathcal{B}_{i}}\mathbb{P}(0<\left|f^{(1)}(X)-f^{(-1)}(X)\right|\le c_{1}w_{i-1}^{\beta},X\in C)\\
 & \le4\bar{c}c_{1}w_{i-1}^{\beta}(t_{i}-t_{i-1})D_{0}\cdot\left[c_{1}w_{i-1}^{\beta}\right]^{\alpha}\\
 & =4D_{0}\bar{c}(t_{i}-t_{i-1})[c_{1}w_{i-1}^{\beta}]^{1+\alpha},
\end{align*}
where the second inequality reuses the bound in (\ref{eq:margin-condition}).

\subsubsection{Last Batch}

For $C\in\mathcal{B}_{M}$, one can similarly obtain 
\begin{align*}
\mathbb{E}[r_{T}^{\textrm{live}}(C)\mathbf{1}(\mathcal{G}_{C})] & \le c_{1}|\mathsf{p}(C)|^{\beta}(T-t_{M-1})\mathbb{P}(0<\left|f^{(1)}(X)-f^{(-1)}(X)\right|\le c_{1}|\mathsf{p}(C)|^{\beta},X\in C).
\end{align*}
Consequently, summing over $C\in\mathcal{B}_{M}$ yields
\begin{align*}
\sum_{C\in\mathcal{B}_{M}}\mathbb{E}[r_{T}^{\textrm{live}}(C)\mathbf{1}(\mathcal{G}_{C})] & \le\sum_{C\in\mathcal{B}_{M}}c_{1}|\mathsf{p}(C)|^{\beta}(T-t_{M-1})\mathbb{P}(0<\left|f^{(1)}(X)-f^{(-1)}(X)\right|\le c_{1}|\mathsf{p}(C)|^{\beta},X\in C)\\
 & \le c_{1}w_{M-1}^{\beta}(T-t_{M-1})D_{0}\cdot\left[c_{1}w_{M-1}^{\beta}\right]^{\alpha}\\
 & =D_{0}(T-t_{M-1})[c_{1}w_{M-1}^{\beta}]^{1+\alpha}.
\end{align*}

\subsection{Putting things together}

In sum, the total regret is bounded by 

\begin{align*}
R_{T}(\pi) & \le c\left(t_{1}+\sum_{i=2}^{M-1}(t_{i}-t_{i-1})\cdot w_{i-1}^{\beta+\alpha\beta}+(T-t_{M-1})w_{M-1}^{\beta+\alpha\beta}\right),
\end{align*}
where $c$ is a constant that depends on $(\alpha,\beta,D,L)$. Recall
that $w_{i}=(\prod_{l=0}^{i-1}g_{l})^{-1}$, and the choices for the
batch size and the split factors~(\ref{eq:batch-bin-size})-(\ref{eq:split-factors}).
We then obtain 
\begin{align*}
t_{1} & \lesssim T^{\frac{1-\gamma}{1-\gamma^{M}}}\log T,\\
(t_{i}-t_{i-1})\cdot w_{i-1}^{\beta+\alpha\beta} & \lesssim T^{\frac{1-\gamma}{1-\gamma^{M}}}\log T,\qquad\text{for }2\leq i\leq M-1,\\
(T-t_{M-1})w_{M-1}^{\beta+\alpha\beta} & \leq Tw_{M-1}^{\beta+\alpha\beta}\lesssim T^{\frac{1-\gamma}{1-\gamma^{M}}}\log T.
\end{align*}
The proof is finished by combining the above three bounds. 

\subsection{Proofs for the clean events}

We are left with proving that the two clean events happen with high
probability. 

\subsubsection{Proof of Lemma~\ref{lemma:clean-event-1}\label{subsec:proof-lemma-clean1}}

Fix the batch index $i$, and a node $C$ in layer-$i$ of the tree
$\mathcal{T}$. By relation (\ref{eq:batch-bin-size}), we have 
\begin{align*}
m_{C,i}^{\star} & =(t_{i}-t_{i-1})\mathbb{P}_{X}(X\in C)\\
 & \asymp|C|^{-(2\beta+d)}\log(T|C|^{d})\mathbb{P}_{X}(X\in C)\\
 & \apprge|C|^{-2\beta}\overset{\mathrm{}}{\ge}g_{0}^{2\beta}\asymp(T^{\frac{1-\gamma}{1-\gamma^{M}}\cdot\frac{2\beta}{2\beta+d}}),
\end{align*}
where the last step uses the fact that $\mathbb{P}_{X}(X\in C)\ge\underline{c}|C|^{d}$.
Therefore, $m_{C,i}^{\star}\ge\frac{3}{4}\log(2T^{2})$ for all $i$
and $C$, as long as $T$ is sufficiently large. This allows us to invoke
Chernoff's bound to obtain that with probability at most $1/T^{2}$
\[
\left|\sum\nolimits_{t=t_{i-1}+1}^{t_{i}}\mathbf{1}\{X_{t}\in C\}-m_{C,i}^{\star}\right|\ge\sqrt{3\log(2T^{2})m_{C,i}^{\star}}.
\]
Denote $E^{c}=\{\exists1\le i\le M,C\in\mathcal{L}_{t_{i-1}+1}\text{ such that }\mid\sum_{t=t_{i-1}+1}^{t_{i}}\mathbf{1}\{X_{t}\in C\}-m_{C,i}^{\star}\mid\ge\sqrt{3\log(2T^{2})m_{C,i}^{\star}}\}$.
Applying union bound to reach 
\begin{align*}
\mathbb{P}(E^{c}) & \le\sum_{C\in\mathcal{T}}\frac{1}{T^{2}}\overset{\mathrm{(i)}}{\le}\frac{1}{T^{2}}\left(\sum_{i=1}^{M}(\prod_{l=0}^{i-1}g_{l})^{d}\right)\overset{\mathrm{(ii)}}{\le}\frac{1}{T^{2}}\cdot M\cdot(\prod_{l=0}^{M-1}g_{l})^{d},
\end{align*}
where step (i) sums over all possible nodes of $\mathcal{T}$ across
batches, and step (ii) is due to $(\prod_{l=0}^{i-1}g_{l})^{d}\le(\prod_{l=0}^{M-1}g_{l})^{d}$
for any $1\le i\le M$. Since $g_{M-1}=1$, we further obtain
\begin{align*}
\mathbb{P}(E^{c}) & \le\frac{1}{T^{2}}\cdot M\cdot(\prod_{l=0}^{M-2}g_{l})^{d}\overset{\mathrm{(iii)}}{\le}\frac{1}{T^{2}}\cdot M\cdot t_{M-1}^{\frac{d}{2\beta+d}}\overset{\mathrm{(iv)}}{\le}D_{1}\frac{1}{T^{2}}\cdot\log T\cdot T^{\frac{d}{2\beta+d}}\le\frac{1}{T},
\end{align*}
where step (iii) invokes relation (\ref{eq:batch-bin-size}), and
step (iv) uses the assumption $M\le D_{1}\log T$. This completes
the proof. 

\subsubsection{Proof of Lemma \ref{lemma:clean-event-2}\label{subsec:proof-lemma-clean2}}

To simplify notation, for any event $F$, we define $\mathbb{P}^{\mathcal{G}_{C}}(F)=\mathbb{P}(E\cap\mathcal{G}_{C}\cap F)$.

Let $\mathcal{D}_{C}^{1}$ be the event that an arm $k\in\underline{\mathcal{I}}_{C}$
is eliminated at the end of batch $i$, and $\mathcal{D}_{C}^{2}$
be the event that an arm $k\notin\bar{\mathcal{I}}_{C}$ is not eliminated
at the end of batch $i$. Consequently, we have
\[
\mathbb{P}^{\mathcal{G}_{C}}(\mathcal{A}_{C}^{c})=\mathbb{P}^{\mathcal{G}_{C}}(\mathcal{D}_{C}^{1})+\mathbb{P}^{\mathcal{G}_{C}}((\mathcal{D}_{C}^{1})^{c}\cap\mathcal{D}_{C}^{2}).
\]
Recall $U(\tau,T,C)=4\sqrt{\frac{\log(2T|C|^{d})}{\tau}}$ . By relation
(\ref{eq:batch-bin-size}), we can write 
\begin{align*}
m_{C,i}^{\star} & =(t_{i}-t_{i-1})\mathbb{P}_{X}(X\in C)\\
 & =l_{i}|C|^{-(2\beta+d)}\log(T|C|^{d})\mathbb{P}_{X}(X\in C),
\end{align*}
where $l_{i}>0$ is a constant chosen such that $U(2m_{C,i}^{\star},T,C)=2c_{0}|C|^{\beta}$.
Under $E$, we have $U(m_{C,i},T,C)\le4c_{0}|C|^{\beta}$ because
$m_{C,i}\ge\frac{1}{2}m_{C,i}^{\star}$.
\begin{enumerate}
\item Upper bounding $\mathbb{P}^{\mathcal{G}_{C}}(\mathcal{D}_{C}^{1})$:
when $\mathcal{D}_{C}^{1}$ occurs, an arm $k\in\underline{\mathcal{I}}_{C}$
is eliminated by some $k^{\prime}\in\mathcal{I}_{\mathsf{p}(C)}^{\prime}$
at the end of batch $i$. This means $\bar{Y}_{C,i}^{(k^{\prime})}-\bar{Y}_{C,i}^{(k)}>U(m_{C,i},T,C)$.
Meanwhile,
\begin{align*}
\bar{f}_{C}^{(k^{\prime})}-\bar{f}_{C}^{(k)} & \le\bar{f}_{C}^{\star}-\bar{f}_{C}^{(k)}\overset{\mathrm{(i)}}{\le}c_{0}|C|^{\beta}\le\frac{1}{2}U(2m_{C,i}^{\star},T,C),
\end{align*}
where step (i) uses the definition of $\underline{\mathcal{I}}_{C}$.
Consequently, ${|\bar{Y}_{C,i}^{(k^{\prime})}-\bar{f}_{C}^{(k^{\prime})}|\le U(m_{C,i},T,C)/4}$
and $|\bar{Y}_{C,i}^{(k)}-\bar{f}_{C}^{(k)}|\le U(m_{C,i},T,C)/4$
cannot hold simultaneously. Otherwise, this would contradict with
$\bar{Y}_{C,i}^{(k^{\prime})}-\bar{Y}_{C,i}^{(k)}>U(m_{C,i},T,C)$
because $m_{C,i}\le2m_{C,i}^{\star}$ under $E$. Therefore,
\begin{align*}
\mathbb{P}^{\mathcal{G}_{C}}(\mathcal{D}_{C}^{1}) & \le\mathbb{P}\left\{ \exists k\in\mathcal{I}_{\mathsf{p}(C)}^{\prime},m_{C,i}\le2m_{C,i}^{\star}:|\bar{Y}_{C,i}^{(k)}-\bar{f}_{C}^{(k)}|\ge\frac{1}{4}U(m_{C,i},T,C)\right\} .
\end{align*}
\item Upper bounding $\mathbb{P}^{\mathcal{G}_{C}}((\mathcal{D}_{C}^{1})^{c}\cap\mathcal{D}_{C}^{2})$:
when $(\mathcal{D}_{C}^{1})^{c}\cap\mathcal{D}_{C}^{2}$ happens,
no arm in $\underline{\mathcal{I}}_{C}$ is eliminated while some
$k\notin\bar{\mathcal{I}}_{C}$ remains in the active arm set. By
definition, there exists $x^{(k)}$ such that ${f^{\star}(x^{(k)})-f^{(k)}(x^{(k)})>8c_{0}|C|^{\beta}}.$
Let $\eta(k)$ be any arm that satisfies $f^{\star}(x^{(k)})=f^{(\eta(k))}(x^{(k)})$,
and one can easily verify $\eta(k)\in\underline{\mathcal{I}}_{C}$.
Since $k$ is not eliminated, we have $\bar{Y}_{C,i}^{(\eta(k))}-\bar{Y}_{C,i}^{(k)}\le U(m_{C,i},T,C)$.
On the other hand, 
\begin{align}
\bar{f}_{C}^{(\eta(k))} & \overset{(\mathrm{iii})}{\ge}f^{(\eta(k))}(x^{(k)})-c_{0}|C|^{\beta}\nonumber \\
 & \ge f^{(k)}(x^{(k)})+8c_{0}|C|^{\beta}-c_{0}|C|^{\beta}\nonumber \\
 & =f^{(k)}(x^{(k)})+7c_{0}|C|^{\beta}\nonumber \\
 & \overset{(\mathrm{iv})}{\ge}\bar{f}_{C}^{(k)}+6c_{0}|C|^{\beta}\ge\bar{f}_{C}^{(k)}+\frac{3}{2}U(m_{C,i},T,C),\label{eq:arm-population-diff}
\end{align}
where steps (iii) and (iv) use Lemma~\ref{lemma:smoothness-avg}.
Inequality (\ref{eq:arm-population-diff}) together with the fact
that $\bar{Y}_{C,i}^{(\eta(k))}-\bar{Y}_{C,i}^{(k)}\le U(m_{C,i},T,C)$
imply $|\bar{Y}_{C,i}^{(k_{0})}-\bar{f}_{C}^{(k_{0})}|\ge U(m_{C,i},T,C)/4$
for either $k_{0}=k$ or $k_{0}=\eta(k)$. Consequently,
\begin{align*}
\mathbb{P}^{\mathcal{G}_{C}}((\mathcal{D}_{C}^{1})^{c}\cap\mathcal{D}_{C}^{2}) & \le\mathbb{P}\left\{ \exists k\in\mathcal{I}_{\mathsf{p}(C)}^{\prime},m_{C,i}\le2m_{C,i}^{\star}:|\bar{Y}_{C,i}^{(k)}-\bar{f}_{C}^{(k)}|\ge\frac{1}{4}U(m_{C,i},T,C)\right\} .
\end{align*}
\end{enumerate}
Combining the two parts we obtain
\begin{align*}
\mathbb{P}^{\mathcal{G}_{C}}(\mathcal{A}_{C}^{c}) & =\mathbb{P}^{\mathcal{G}_{C}}(\mathcal{D}_{C}^{1})+\mathbb{P}^{\mathcal{G}_{C}}((\mathcal{D}_{C}^{1})^{c}\cap\mathcal{D}_{C}^{2})\\
 & \le2\cdot\mathbb{P}\left\{ \exists k\in\mathcal{I}_{\mathsf{p}(C)}^{\prime},m_{C,i}\le2m_{C,i}^{\star}:|\bar{Y}_{C,i}^{(k)}-\bar{f}_{C}^{(k)}|\ge\frac{1}{4}U(m_{C,i},T,C)\right\} \\
 & \le\frac{4m_{C,i}^{\star}}{T|C|^{d}},
\end{align*}
where the last inequality applies Lemma~\ref{lemma:reward-concentration}.

\subsection{Auxiliary lemmas}

\begin{lemma}\label{lemma:reward-concentration}For any $1\le i\le M-1$
and $C\in\mathcal{B}_{i}$, one has
\[
\mathbb{P}\left\{ \exists k\in\mathcal{I}_{\mathsf{p}(C)}^{\prime},m_{C,i}\le2m_{C,i}^{\star}:|\bar{Y}_{C,i}^{(k)}-\bar{f}_{C}^{(k)}|\ge\frac{1}{4}U(m_{C,i},T,C)\right\} \le\frac{2m_{C,i}^{\star}}{T|C|^{d}}.
\]

\end{lemma}
\begin{proof}
Recall in Algorithm~\ref{alg:adaptive-bin} we pull each arm in a
round-robin fashion within a bin during batch $i$. Fix $\tau>0$.
Let $\bar{Y}_{\tau}^{(k)}=\sum_{j=1}^{\tau}Y_{j}^{(k)}/\tau$ where
$Y_{j}^{(k)}$'s are i.i.d.~random variables with $Y_{j}^{(k)}\in[0,1]$
and $\mathbb{E}[Y_{j}^{(k)}]=\bar{f}_{C}^{(k)}$. By Hoeffding's inequality,
with probability $1/(T|C|^{d})$, we have
\[
|\bar{Y}_{\tau}^{(k)}-\bar{f}_{C}^{(k)}|\ge\sqrt{\frac{\log(2T|C|^{d})}{2\tau}}.
\]
Applying union bound to get
\[
\mathbb{P}\left\{ \exists k\in\mathcal{I}_{\mathsf{p}(C)},0\le\tau\le m_{C,i}^{\star}:|\bar{Y}_{\tau}^{(k)}-\bar{f}_{C}^{(k)}|\ge\sqrt{\frac{\log(2T|C|^{d})}{2\tau}}\right\} \le\frac{2m_{C,i}^{\star}}{T|C|^{d}},
\]
which completes the proof.
\end{proof}
\begin{lemma}\label{lemma:smoothness-avg}Fix $k\in\{1,-1\}$ and
$C\in\mathcal{T}$, for any $x\in C$, one has
\[
|\bar{f}_{C}^{(k)}-f^{(k)}(x)|\le c_{0}|C|^{\beta},
\]

where $c_{0}=2Ld^{\beta/2}+1$.

\end{lemma}
\begin{proof}
For notation simplicity, we write $f$ for $f^{(k)}$ in the following
proof. By definition,
\begin{align*}
|\bar{f}_{C}-f(x)| & =|\frac{1}{\mathbb{P}(C)}\int_{C}(f(y)-f(x))\mathsf{d}\mathbb{P}(y)|\\
 & \le\frac{1}{\mathbb{P}(C)}\int_{C}|f(y)-f(x)|\mathsf{d}\mathbb{P}(y)\\
 & \le\frac{1}{\mathbb{P}(C)}\int_{C}L\|x-y\|_{2}^{\beta}\mathsf{d}\mathbb{P}(y),
\end{align*}
where the first inequality uses the triangle inequality, and the second
inequality is due to the smoothness condition. Since $x\in C$, we
further have
\begin{align*}
|\bar{f}_{C}-f(x)| & \le\frac{1}{\mathbb{P}(C)}\int_{C}L\|x-y\|_{2}^{\beta}\mathsf{d}\mathbb{P}(y)\\
 & \le\frac{1}{\mathbb{P}(C)}\int_{C}Ld^{\beta/2}|C|^{\beta}\mathsf{d}\mathbb{P}(y)\\
 & \le c_{0}|C|^{\beta}.
\end{align*}
This completes the proof. 
\end{proof}

\section{Proof of Theorem~\ref{thm:static-lower} \label{sec:Proof-of-Theorem-failure}}

As we argued after the statement of Theorem~\ref{thm:static-lower},
one needs to set $t_{1}\asymp T^{9/19}$, and $t_{2}\asymp T^{15/19}$.
Therefore, throughout the proof, we assume this is true and only focus
on the number $g$ of bins. 

To construct a hard instance, we partition $[0,1]$ into $z$ bins
with equal width. Denote the bins by $C_{j}$ for $j=1,...,z$, and
let $q_{j}$ be the center of $C_{j}$. Define a function $\phi:[0,1]\mapsto\mathbb{R}$
as $\phi(x)=(1-|x|)\mathbf{1}\{|x|\le1\}.$ Correspondingly define
a function $\varphi_{j}:[0,1]\mapsto\mathbb{R}$ as $\varphi_{j}(x)=D_{\phi}z^{-1}\phi(2z(x-q_{j}))\mathbf{1}\{x\in C_{j}\},$
where $D_{\phi}=\min(2^{-1}L,1/4)$. Define a function $f:[0,1]\mapsto\mathbb{R}$:
\[
f(x)=\frac{1}{2}+\varphi_{1}(x).
\]
The problem instance of interest is $v=(f^{(1)}(x)=f(x),f^{(-1)}(x)=\frac{1}{2})$.
It is easy to verify $v\in\mathcal{F}(1,1).$ Throughout the proof,
we condition on the event $E$ specified by Lemma~\ref{lemma:clean-event-1},
which says the number of samples allocated to a bin concentrates well
around its expectation. We will show even under this good event, there
exists a choice of $z$ that makes successive elimination fail to
remove the suboptimal arms at the end of a batch with constant probability. 

\subsection{A helper lemma}

We begin with presenting a helper lemma that will be used extensively
in the later part of the proof. The claim is intuitive: if the sample
size is small, it is not sufficient to tell apart two Bernoulli distributions
with similar means. Then, in our context, arm elimination will not
occur. 

\begin{lemma}\label{lemma:elimination-failure}Assume $m_{B,i}\le2m_{B,i}^{\star}$.
For any $B\subseteq[0,1]$ and $i\in\{1,2\}$. If $\bar{f}_{B}^{(1)}-\bar{f}_{B}^{(-1)}\le\delta\le1/\sqrt{m_{B,i}^{\star}}$
for some $\delta>0$ , then
\[
\mathbb{P}\left(\bar{Y}_{B,i}^{(1)}-\bar{Y}_{B,i}^{(-1)}>U(m_{B,i},T,B)\right)\le\frac{t_{i}}{T}.
\]

\end{lemma}

\begin{proof}Fix $0<\tau\le m_{B,i}^{\star}$. Let $\bar{Y}_{\tau}^{(k)}=\sum_{l=1}^{\tau}Y_{l}^{(k)}/\tau$
where $Y_{l}^{(k)}$'s are i.i.d.~random variables with $Y_{l}^{(k)}\in[0,1]$
and $\mathbb{E}[Y_{l}^{(k)}]=\bar{f}_{B}^{(k)}$ for $k\in\{1,-1\}$.
Recall $U(\tau,T,B)=4\sqrt{\frac{\log(2T|B|)}{\tau}}$\footnote{We
remark the constant 4 is not essential for the proof to work. For
any $c>0$ , $c\log(2T|B|)=\log((2T|B|)^{c})$ so the final success
probability is still tiny as long as $T$ is sufficiently large.}.
Then, 
\begin{align*}
\mathbb{P}\left(\bar{Y}_{\tau}^{(1)}-\bar{Y}_{\tau}^{(-1)}>U(2\tau,T,B)\right) & \overset{\mathrm{(i)}}{\le}\mathbb{P}\left(\bar{Y}_{\tau}^{(1)}-\bar{Y}_{\tau}^{(-1)}>\delta+\sqrt{\frac{\log(2T/g)}{2\tau}}\right)\\
 & \overset{\mathrm{(ii)}}{\le}\mathbb{P}\left(\bar{Y}_{\tau}^{(1)}-\bar{Y}_{\tau}^{(-1)}>\bar{f}_{B}^{(1)}-\bar{f}_{B}^{(-1)}+\sqrt{\frac{\log(2T/g)}{2\tau}}\right)\\
 & \overset{\mathrm{(iii)}}{\le}\frac{g}{T},
\end{align*}
where step (i) is because $\delta\le1/\sqrt{m_{B,i}^{\star}}\le1/\sqrt{\tau}$,
step (ii) is due to $\bar{f}_{B}^{(1)}-\bar{f}_{B}^{(-1)}\le\delta$,
and step (iii) uses Hoeffding's inequality. Applying union bound to
get
\[
\mathbb{P}\left(\exists0<\tau\le m_{B,i}^{\star}:\bar{Y}_{\tau}^{(1)}-\bar{Y}_{\tau}^{(-1)}>U(2\tau,T,B)\right)\le\frac{m_{B,i}^{\star}g}{T}\le\frac{t_{i}}{T}.
\]
This finishes the proof. \end{proof}

\subsection{Three failure cases for $g$}

Fix some small constant $\varepsilon>0$ to be specified later. From
now on, we use $\hat{\pi}$ to denote $\hat{\pi}_{\mathrm{static}}$
for simplicity. We split the proof into three cases: (1) $g\geq T^{3/19+\varepsilon}$;
(2) $g\leq T^{3/19-\varepsilon}$; (3) and $g\in(T^{3/19-\varepsilon},T^{3/19+\varepsilon})$. 

\paragraph{Case 1: $g\geq T^{3/19+\varepsilon}$.}

Set $z=T^{3/19-\varepsilon/2}$. Assume without loss of generality
that $g=H\cdot z$ for some $H\geq4$; see Figure~\ref{fig:N_ge_g}
for an illustration of the instance. Suppose $C_{1}=\cup_{l=1}^{H}B_{l}$,
where $B_{l}$'s are the bins produced by $\hat{\pi}$ that lie in
$C_{1}$. It is clear that 
\begin{align}
R_{T}(\hat{\pi}) & \overset{\mathrm{(i)}}{\ge}\mathbb{E}\left[\sum_{t=t_{1}+1}^{t_{2}}\left(f^{\star}(X_{t})-f^{\hat{\pi}_{t}(X_{t})}(X_{t})\right)\right]\nonumber \\
 & \overset{\mathrm{(ii)}}{=}\mathbb{E}\left[\sum_{t=t_{1}+1}^{t_{2}}\left(f^{\star}(X_{t})-f^{\hat{\pi}_{t}(X_{t})}(X_{t})\right)\mathbf{1}\{X_{t}\in C_{1}\}\right]\nonumber \\
 & \overset{\mathrm{(iii)}}{\geq}\sum_{t=t_{1}+1}^{t_{2}}\sum_{l=H/4}^{3H/4}\mathbb{E}\left[\left(f^{\star}(X_{t})-f^{\hat{\pi}_{t}(X_{t})}(X_{t})\right)\mathbf{1}\{X_{t}\in B_{l}\}\right],\label{eq:reg-case-1}
\end{align}
where step (i) is because the total regret is greater than the regret
incurred during the second batch, step (ii) uses the fact that under
the instance $v$, the mean rewards of the two arms differ only in
$C_{1}$, and step (iii) arises since $C_{1}=\cup_{l=1}^{H}B_{l}$.
Now we turn to lower bounding $\mathbb{E}\left[\left(f^{\star}(X_{t})-f^{\hat{\pi}_{t}(X_{t})}(X_{t})\right)\mathbf{1}\{X_{t}\in B_{l}\}\right]$
for each $H/4\leq l\leq3H/4$. 

Consider any such $B_{l}$. We drop the subscripts and write $B$
instead for simplicity. By the design of $v$, we have $\bar{f}_{B}^{(1)}-\bar{f}_{B}^{(-1)}\le D_{\phi}z^{-1}=\delta$,
which obeys $D_{\phi}z^{-1}\le1/\sqrt{m_{B,1}^{\star}}$---a consequence
of the choice of $z$. Additionally, we have $m_{B,1}\le2m_{B,1}^{\star}$
under $E$. Therefore, we can invoke Lemma~\ref{lemma:elimination-failure}
to obtain
\[
\mathbb{P}\left(\bar{Y}_{B,1}^{(1)}-\bar{Y}_{B,1}^{(-1)}>U(m_{B,1},T,B)\right)\le\frac{t_{1}}{T}\leq\frac{1}{2}.
\]
In words, with probability exceeding $1/2$, no elimination will happen
for the bin $B$. As a result, we obtain 
\begin{align*}
R_{T}(\hat{\pi}) & \ge\sum_{t=t_{1}+1}^{t_{2}}\sum_{l=H/4}^{3H/4}\mathbb{E}\left[\left(f^{\star}(X_{t})-f^{\hat{\pi}_{t}(X_{t})}(X_{t})\right)\mathbf{1}\{X_{t}\in B_{l}\}\right]\\
 & \gtrsim H\cdot\frac{t_{2}}{g}\cdot z^{-1}\asymp\frac{t_{2}}{z^{2}}=T^{\frac{9}{19}+\epsilon},
\end{align*}
where we have used the choice of $z$. So Theorem~\ref{thm:static-lower}
holds with $\kappa=\epsilon$.

\paragraph{Case 2: $g\leq T^{3/19-\varepsilon}$.}

Set $z=T^{3/19-\varepsilon/8}$. We have $g<z$ and there exists $H>1$
such that $z=H\cdot g$; see Figure~\ref{fig:N_ll_g} for an illustration
of the instance. Let $B$ be the bin produced by $\hat{\pi}$ such
that $C_{1}\subset B$. By the design of $v$, we have 
\[
\bar{f}_{B}^{(1)}-\bar{f}_{B}^{(-1)}\le\frac{1}{H}(1/2+D_{\phi}z^{-1})+(1-\frac{1}{H})\frac{1}{2}-\frac{1}{2}=\frac{D_{\phi}z^{-1}}{H}.
\]
Let $\delta=\frac{D_{\phi}z^{-1}}{H}$, we have $\delta\le1/\sqrt{m_{B,1}^{\star}}$
due to our choice of $z$. Additionally, we have $m_{B,1}\le2m_{B,1}^{\star}$
under $E$. Therefore, we can invoke Lemma~\ref{lemma:elimination-failure}
to obtain
\[
\mathbb{P}\left(\bar{Y}_{B,1}^{(1)}-\bar{Y}_{B,1}^{(-1)}>U(m_{B,1},T,B)\right)\le\frac{t_{1}}{T}\le\frac{1}{2}.
\]
Thus, with probability exceeding $1/2$, the suboptimal arm is not
eliminated in $B$. Similar to the previous case, we obtain

\begin{align*}
R_{T}(\hat{\pi}) & \overset{\mathrm{}}{\ge}\mathbb{E}\left[\sum_{t=t_{1}+1}^{t_{2}}\left(f^{\star}(X_{t})-f^{\hat{\pi}_{t}(X_{t})}(X_{t})\right)\right]\\
 & \overset{\mathrm{}}{=}\mathbb{E}\left[\sum_{t=t_{1}+1}^{t_{2}}\left(f^{\star}(X_{t})-f^{\hat{\pi}_{t}(X_{t})}(X_{t})\right)\mathbf{1}\{X_{t}\in C_{1}\}\right]\\
 & \apprge\frac{t_{2}}{z^{2}}\overset{\mathrm{}}{=}T^{\frac{9}{19}+\frac{\epsilon}{4}}.
\end{align*}
So Theorem~\ref{thm:static-lower} holds with $\kappa=\epsilon/4$.

\paragraph{Case 3: $g\in(T^{3/19-\varepsilon},T^{3/19+\varepsilon})$.}

Set $z\asymp T^{1/4}$. We then have $g<z$, as long as $\varepsilon\leq1/19$.
And there exists $H>1$ such that $z=H\cdot g$; see Figure~\ref{fig:N_ll_g}
for an illustration of the instance. Let $B$ be the bin produced
by $\hat{\pi}$ such that $C_{1}\subset B$. By the design of $v$,
we have 
\[
\bar{f}_{B}^{(1)}-\bar{f}_{B}^{(-1)}\le\frac{1}{H}(1/2+D_{\phi}z^{-1})+(1-\frac{1}{H})\frac{1}{2}-\frac{1}{2}=\frac{D_{\phi}z^{-1}}{H}.
\]
Let $\delta=\frac{D_{\phi}z^{-1}}{H}$, we have $\delta\le1/\sqrt{m_{B,1}^{\star}}$
due to our choice of $z$. Additionally, we have $m_{B,1}\le2m_{B,1}^{\star}$
under $E$. Therefore, we can invoke Lemma~\ref{lemma:elimination-failure}
to obtain
\[
\mathbb{P}\left(\bar{Y}_{B,1}^{(1)}-\bar{Y}_{B,1}^{(-1)}>U(m_{B,1},T,B)\right)\le\frac{t_{1}}{T}\le\frac{1}{4}.
\]
This means with probability at least $3/4$, arm elimination does
not occur in $B$ after the first batch. Moreover, since $\delta\le1/\sqrt{m_{B,2}^{\star}}$
by the choice of $z$, and $m_{B,2}\le2m_{B,2}^{\star}$ under $E$,
we can apply Lemma~\ref{lemma:elimination-failure} again to get

\[
\mathbb{P}\left(\bar{Y}_{B,2}^{(1)}-\bar{Y}_{B,2}^{(-1)}>U(m_{B,2},T,B)\right)\le\frac{t_{2}}{T}\le\frac{1}{4}.
\]
In all, with probability at least $1/2$, arm elimination does not
occur in $B$ after the second batch. Similar to before, we reach
the conclusion that 
\begin{align*}
R_{T}(\hat{\pi}) & \overset{\mathrm{}}{\ge}\mathbb{E}\left[\sum_{t=t_{2}+1}^{T}\left(f^{\star}(X_{t})-f^{\hat{\pi}_{t}(X_{t})}(X_{t})\right)\right]\\
 & \overset{\mathrm{}}{=}\mathbb{E}\left[\sum_{t=t_{2}+1}^{T}(f^{\star}(X_{t})-f^{\hat{\pi}_{t}(X_{t})}(X_{t}))\mathbf{1}\{X_{t}\in C_{1}\}\right]\\
 & \overset{\mathrm{}}{\gtrsim}\frac{T}{z^{2}}\overset{\mathrm{}}{=}T^{\frac{1}{2}}.
\end{align*}
We see that Theorem~\ref{thm:static-lower} holds with $\kappa=1/38$.

\section{Proof of Theorem~\ref{thm:alpha-lower}}\label{sec:alpha-lower-proof}
% When $\beta=d=1$, 
Recall $\gamma(\alpha^\star)=\beta(\alpha^\star+1)/(2\beta+d)=(\alpha^\star+1)/3=\gamma^\star$. Denote by $h_M(\gamma^\star)=(1-\gamma^\star)/(1-(\gamma^\star)^M)$. Fix some small constant $\epsilon>0$. We first deal with $M=2$ and then extend the analysis to general $M$. Throughout the proof, we use the fact that the algorithm needs to provide its first batch size $t_1$ prior to the game, and design $\alpha^\star$ such that any choice of $t_1$ would fail.

\subsection{When $M=2$}
Under $M=2$,  the theoretical optimal rate is $T^{(1-\gamma(\alpha^\star))/(1-\gamma(\alpha^\star)^2)}=T^{3/(\alpha^\star+4)}$. 
% Since fixed grid and adaptive grid are the same when $M=2$, by relation~(\ref{eq:lower-bound-final}) we know  $T^{3/5}\le t_1<T^{3/4}$ because any $t_1$ that falls outside this range is clearly suboptimal.
\paragraph{Case of  $t_1=T^{3/5+\epsilon}$.} Take $\alpha^\star=1$. By relation~(\ref{eq:lower-bound-final}), we have
\[\sup_{f\in\mathcal{F}(\alpha^\star,\beta)}
R_{T}(\pi)
\gtrsim 
\max\left\{t_1,\frac{T}{t_1^{\frac{\alpha^\star+1}{3}}}\right\}
\ge t_1=T^{\frac{3}{5}+\epsilon}.\]
% for some $\kappa_1>0$.

\paragraph{Case of  $t_1=T^{3/4-\epsilon}$.} Take $\alpha^\star>0$ such that $T^{3/(\alpha^\star+4)}>t_1$. By relation~(\ref{eq:lower-bound-final}), we have
\[\sup_{f\in\mathcal{F}(\alpha^\star,\beta)}
R_{T}(\pi)
\gtrsim \max\left\{t_1,\frac{T}{t_1^{\frac{\alpha^\star+1}{3}}}\right\}
\ge \frac{T}{t_1^{\frac{\alpha^\star+1}{3}}}
=T^{\frac{3}{\alpha^\star+4}
+(\alpha^\star+1)(\frac{\epsilon}{3}+\frac{1}{\alpha^\star+4}-\frac{1}{4})},\]
% where $\kappa_1=(\alpha^\star+1)\epsilon/3-\alpha^\star/4>0$.
where $\frac{\epsilon}{3}+\frac{1}{\alpha^\star+4}-\frac{1}{4}>0$ due to the choice of $\alpha^\star$.

\subsection{General $M$}
Since $(1-\gamma(\alpha))/(1-\gamma^M(\alpha))$ is decreasing as $\alpha$ increases, one has 
\[\frac{1-\gamma(1)}{1-\gamma^M(1)}
<\frac{1-\gamma(0)}{1-\gamma^M(0)}.\]
We can divide the proof into the below two cases.
\paragraph{Case of  $t_1=T^{(1-\gamma(1))/(1-\gamma^M(1))+\epsilon}$.} Take $\alpha^\star=1$. During the first batch, the learner can do no better than pull an arm uniformly at random. We can use the instance $(f_1(x)=1,f_2(x)=0)$ so that
\[\sup_{f\in\mathcal{F}(\alpha^\star,\beta)}
R_{T}(\pi)
\gtrsim t_1=T^{\frac{1-\gamma^\star}{1-(\gamma^{\star})^M} +\epsilon}.\]
% for some $\epsilon>0$.

\paragraph{Case of  $t_1=T^{(1-\gamma(0))/(1-\gamma^M(0))-\epsilon}$.} Take $\alpha^\star>0$ such that  $T_1=T^{(1-\gamma^\star)/(1-(\gamma^\star)^{M})}>t_1$. Define 
\begin{equation*}
    T_2=T^{\frac{1-\gamma^\star}{1-(\gamma^\star)^{M-1}}}\cdot
    t_1^{\frac{\gamma^\star-(\gamma^\star)^{M-1}}{1-(\gamma^\star)^{M-1}}},
\end{equation*}
and
\begin{equation*}
    T_i=t_1^{\frac{\gamma^\star-(\gamma^\star)^{i-1}}{\gamma^\star-1}}\cdot
    T_2^{\frac{(\gamma^\star)^{i-1}-1}{\gamma^\star-1}}
\end{equation*}
when $3\le i\le M$. It can be verified that $T_M=T$. Define the events
\[A_{i}=\{t_{i-1}<T_{i-1}<T_{i}\le t_{i}\}.\]
for $2\le i\le M$. Moreover, take $z_{2}=\lceil(36t_{1}(M-1)^{2})^{1/(2\beta+d)}\rceil$ and $z_{i}=\lceil(36T_{i-1}(M-1)^{2})^{1/(2\beta+d)}\rceil$ for $3\le i\le M$. By relation~\eqref{eq:reg-to-m}
, one has
\begin{align*}
    \sup_{(f,\frac{1}{2})\in\mathcal{F}(\alpha^\star,\beta)}R_{T}(\pi;f) &\apprge M^{-\frac{(\alpha\beta+1)(1+\alpha)}{\alpha}}T_{i}z_{i}^{-\beta(1+\alpha)}\left(\mathbb{E}_{\omega\sim\omgdist}\mathbb{P}_{\pi,\omega}(A_{m})-\frac{1}{2(M-1)}\right)^{\frac{1+\alpha}{\alpha}},
\end{align*}
for any $2\le i\le M$. Since there exists $i^\star\in\{2,...,M\}$ such that $\mathbb{E}_{\omega\sim\omgdist}\mathbb{P}_{\pi,\omega}(A_{i^\star})\ge1/(M-1)$, we reach
\begin{align*}
    \sup_{(f,\frac{1}{2})\in\mathcal{F}(\alpha^\star,\beta)}R_{T}(\pi;f) &\apprge M^{-\frac{(\alpha\beta+1)(1+\alpha)}{\alpha}}T_{i^\star}z_{i^\star}^{-\beta(1+\alpha)}\left(\mathbb{E}_{\omega\sim\omgdist}\mathbb{P}_{\pi,\omega}(A_{i^\star})-\frac{1}{2(M-1)}\right)^{\frac{1+\alpha}{\alpha}}\\
    &\apprge M^{-\frac{(\alpha\beta+2)(1+\alpha)}{\alpha}}T_{i^\star}z_{i^\star}^{-\beta(1+\alpha)}\\
    &\asymp T^{\frac{1-\gamma^\star}{1-(\gamma^{\star})^M} \;+\; c_2(\gamma^{\star})^{M-1}\cdot\frac{1-\gamma^\star}{1-(\gamma^{\star} )^{M-1}}
    }\cdot(\frac{1}{M})^{c_3},
\end{align*}
for some $c_2, c_3>0$.

\end{document}